\documentclass[11pt, reqno]{amsart}

\usepackage{lineno, hyperref}

\modulolinenumbers[5]
\usepackage{diagbox}
\usepackage{amsmath}
\usepackage{amssymb}
\usepackage{mathrsfs}
\usepackage{amsthm}
\usepackage{bm}
\usepackage{bbm}
\usepackage{graphicx}
\usepackage{subfigure}
\usepackage{xcolor} 
\usepackage{tikz} 
\usetikzlibrary{arrows, shapes, chains} 
\usepackage{booktabs}
\usepackage{float}
\usepackage{geometry}
\usepackage{color}

\newtheorem{Def}{Definition}[section]
\newtheorem{lem}[Def]{Lemma}
\newtheorem{theo}[Def]{Theorem}
\newtheorem{pro}[Def]{Proposition}
\newtheorem{rem}[Def]{Remark}
\newtheorem{ex}[Def]{Example}
\newtheorem{assum}{Assumption}
\newtheorem{cor}[Def]{Corollary}
\usepackage{pgf}
\definecolor{Green}{RGB}{0, 128, 0}

\newcommand{\LL}{\langle}
\newcommand{\RR}{\rangle}
\newcommand{\R}{\mathbb R}
\newcommand{\mbb}{\mathbb}
\newcommand{\mbf}{\mathbf}
\newcommand{\ud}{\mathrm d}
\numberwithin{equation}{section}
\allowdisplaybreaks \allowdisplaybreaks[4]

\begin{document}

	\title[Asymptotic error distribution for tamed Euler method]{Asymptotic error distribution for tamed Euler method with coupled monotonicity condition}

	\author{Xinjie Dai}
	\address{School of Mathematics and Statistics, Yunnan University, Kunming 650500, Yunnan, China}
	\email{dxj@ynu.edu.cn}
	
	\author{Diancong Jin}
\address{School of Mathematics and Statistics, Huazhong University of Science and Technology, Wuhan 430074, China;
	Hubei Key Laboratory of Engineering Modeling and Scientific Computing, Huazhong University of Science and Technology, Wuhan 430074, China}
\email{jindc@hust.edu.cn}

	\author{Jiaoyang Xu}
\address{School of Mathematics and Statistics, Huazhong University of Science and Technology, Wuhan 430074, China}
\email{xujiaoyang@hust.edu.cn (Corresponding author)}

	\thanks{This work is supported by National Natural Science Foundation of China (Nos.\ 12401547, 12201228, 12471391), and Yunnan Fundamental Research Project (No.\ 202501AU070074).
	}


	\begin{abstract}
		This paper establishes the asymptotic error distribution of the tamed Euler method for stochastic differential equations (SDEs) with a coupled monotonicity condition, that is, the limit distribution of the corresponding normalized error process. Specifically, for SDEs driven by multiplicative noise, we first propose a tamed Euler method parameterized by $\alpha\in (0, 1]$ and establish that its strong convergence rate is $\alpha\wedge\frac{1}{2}$. Notably, $\alpha $ can take arbitrary positive values by adjusting the regularization coefficient without altering the strong convergence rate. We then derive the asymptotic error distribution for this tamed Euler method. Further, we infer from the limit equation that among the tamed Euler method of strong order $\frac{1}{2}$, the one with $\alpha = \frac{1}{2}$ yields the largest mean-square error after a long time, while those of $\alpha>\frac{1}{2}$ share a unified asymptotic error distribution. In addition, our analysis is also extended to SDEs with additive noise and similar conclusions are obtained. Additional treatments are required to accommodate super-linearly growing coefficients, a feature that distinguishes our analysis on the asymptotic error distribution from established results.
	\end{abstract}

	\maketitle


\section{Introduction}

In this paper, we consider the tamed Euler method of the following stochastic differential equation (SDE):\ 
\begin{align} \label{SDE1}
\begin{cases}
 \ud X(t) = f(X(t)) \ud t + g(X(t)) \ud \mbf W(t), \qquad t \in (0, T], \\
X(0) = X_0, 
\end{cases}
\end{align}
where $\mbf W(t) = (W_1(t), \ldots, W_m(t))^\top_{t\in[0, T]}$ is an $m$-dimensional standard Brownian motion defined on $(\Omega, \mathcal F, \mbf P)$ with respect to (w.r.t.) a filtration $\{\mathcal F_t\}_{t\in[0, T]}$ with $\{\mathcal F_t\}_{t\in[0, T]}$ satisfying the usual conditions. By imposing a coupled monotonicity condition, we allow the coefficients $f$ and $g$ to exhibit super-linear growth, as specified in Assumption \ref{assum1}. This relaxed regularity requirement enables \eqref{SDE1} to encompass various important models subject to stochastic perturbations, such as the financial 3/2-model and the stochastic Ginzburg--Landau equation.

Over the past decades, stochastic numerical methods have evolved into an indispensable tool for the numerical approximation of SDEs and for the characterization of their intrinsic dynamical behaviors, thereby attracting extensive attention and spurring intensive research initiatives (cf.\ \cite{Kloeden1992, Milsteinbook}). Among various numerical methods, the Euler method stands out due to its simplicity and computational efficiency, making it a commonly used tool for simulating SDEs in real-world applications. However, the performance of the Euler method is highly dependent on the regularity of coefficients of SDEs:\ the global Lipschitz condition, which was often assumed in early studies to ensure convergence, frequently fails to hold for most models (cf.\ \cite{Huzen2011}).

To overcome the inherent divergence of the standard Euler scheme under the super-linear growth condition, several modified explicit schemes have been developed. Notable examples include the stopped Euler method, the truncated Euler method, and the tamed Euler method. By introducing a regularization term to suppress growing rapidly coefficients, the tamed Euler method effectively avoids finite-time explosion and ensures strong convergence for SDEs with non-globally Lipschitz coefficients. Concerning this kind of method, significant progress has been made in the convergence analysis:\ \cite{Hutzen2012} established the strong convergence rate $\frac{1}{2}$ of the numerical method applied to \eqref{SDE1}, under the assumption that the drift coefficient satisfies the one-sided Lipschitz condition and the diffusion coefficient is globally Lipschitz continuous. Further, \cite{Sabanis2016} extended this result to non-autonomous SDEs with a coupled monotonicity condition.

Compared with the abundant results on convergence analysis, research on probabilistic limit theorems for errors between numerical and exact solutions has been far from adequate. A core concept in this under-explored area is the asymptotic error distribution:\ the limiting distribution of the normalized error process, as the step-size vanishes. This distribution not only serves as a benchmark for determining the optimal strong convergence of numerical methods but also characterizes how the error distribution evolves in the small step-size regime. Beyond its theoretical importance, the asymptotic error distribution plays a vital role in analyzing error structures (cf.\ \cite{Bouleau}) and optimizing key parameters for stochastic algorithms (cf.\ \cite{CLTMonte15, Jin2025}). In this field, the pioneering work \cite{Protter1998AOP} derived the asymptotic error distribution for the Euler method under the global Lipschitz condition, proving that the normalized error process converges weakly to a limit process governed by a linear SDE. Subsequently \cite{Protter2020SPA} extended this result to SDEs with locally Lipschitz coefficients. Recently, \cite{SDR2025} established the asymptotic error distribution of $\theta$-method for stochastic Hamiltonian systems with additive noise. Moreover, the asymptotic error distribution of stochastic Runge--Kutta method of strong order $1$, applied to Stratonovich-type SDEs with both additive and multiplicative noise, was established in \cite{Jin2025}.

To the best of our knowledge, the asymptotic error distribution for \eqref{SDE1} with a coupled monotonicity condition remains an unstudied topic in the current literature. This paper is dedicated to filling this gap. Specifically, we focus on presenting the asymptotic error distribution of the tamed Euler method applied to \eqref{SDE1} with a coupled monotonicity condition. In general, when investigating the asymptotic error distribution of a numerical method for \eqref{SDE1}, one appropriately decomposes the normalized error process into several dominant terms and remainder terms. Owing to the super-linear growth of the coefficients, it is difficult to prove that the auxiliary process, obtained by neglecting the remainder terms in the aforementioned expansion, has the same distribution as the normalized error process, a common treatment in the existing literature (cf.\ \cite{SDR2025, Jin2025}). To solve this problem, we analyze the convergence of both the dominant terms and the remainder terms in the space $\mbf C([0, T];\mbb R^d)$. Moreover, for the weak convergence of stochastic integrals in dominant terms, Jacod’s theory on convergence in distribution of conditional Gaussian martingales will play a key role (cf.\ \cite{Fukasawa2023, Jacod97, Protter1998AOP}). Further, the presence of multiplicative noise makes the normalized error process not continuous w.r.t.\ the Brownian motion, as a solution mapping. We resolve it by means of \cite[Theorem 3.2]{HJWY24}, a uniform approximation theorem for convergence in distribution, where the normalized error process is again discretized by the tamed Euler method in view of the super-linear growth condition; see Lemma \ref{A1}.

Our main results, the asymptotic error distribution for the tamed Euler method, include Theorem \ref{maintheorem1} for the case of multiplicative noise, and Theorem \ref{maintheorem2} for the case of additive noise. Concerning the multiplicative noise, we first propose a tamed Euler method \eqref{TEM} and establish its strong convergence order of $\alpha\wedge\frac{1}{2}$; see Theorem \ref{TEMconverge}. Then Theorem \ref{M2} indicates that $\alpha $ can take arbitrary positive values by adjusting the regularization coefficient without altering the strong convergence rate, which successfully extends the range of $\alpha$ from $(0, \frac{1}{2}]$ in \cite{Sabanis2016} to $(0, \infty) $. Further, we derive the asymptotic error distribution of the tamed Euler method \eqref{TEM}, i.e., $N^{\frac{1}{2}\wedge\alpha}(\bar{X}^\alpha_N-X(T)) \overset{d}{\Longrightarrow} U^\alpha(T)$ with $\{U^\alpha(t)\}_{t\in[0, T]}$ satisfying a linear SDE; see Theorem \ref{maintheorem1}. In addition, Corollary \ref{Cor1} infers that among the tamed Euler method \eqref{TEM} of strong order $\frac{1}{2}$, the one with $\alpha = \frac{1}{2}$ has the largest mean-square error after a long time; see also Remark \ref{RA}. Moreover, we prove in Section \ref{Secadd} that similar conclusions hold for the tamed Euler method \eqref{TEMA} applied to SDEs with additive noise.

Even though the error analysis for additive noise is generally more tractable compared with multiplicative noise, the asymptotic error distribution may not be easier to handle. This is because the numerical method has a higher strong convergence order in the case of additive noise, which requires expanding the error process at least to the first-order terms. Meanwhile, when investigating the asymptotic error distribution for the numerical method for additive noise, we encounter a new problem:\ For $\alpha \ge 1$, it appears that the $\mbf L^p$-convergence of the terms $J_2^{\alpha, N}$ and $J_3^{\alpha, N}$, which originate from the expansion of the normalized error process (cf.\ \eqref{YY0}), can currently be established only for fixed time points. To address this issue, we employ a different argument to establish their convergence in probability in the space $\mbf C([0, T];\R^d)$ in Lemma \ref{J2C}, thereby successfully deriving the asymptotic error distribution.

We summarize the primary contributions of this work as follows:\ (1) We introduce a parameterized tamed Euler method where the regularization parameter $\alpha$ is permitted to be arbitrarily large while maintaining a tractable framework for convergence analysis. (2) We establish the asymptotic error distribution for the proposed scheme under a coupled monotonicity condition, thereby accommodating super-linearly growing diffusion coefficients for the first time in the literature.

This paper consists of six sections. Section \ref{Sec2} introduces some notations and the definition of stable convergence in distribution. In Section \ref{Sec3}, under a coupled monotonicity condition, we present a tamed Euler method for multiplicative noise and give its strong convergence rate. Then Sections \ref{Sec4} and \ref{Secadd} establish the asymptotic error distribution of the tamed Euler method for the cases of multiplicative noise and additive noise, respectively. In Section \ref{Sec5}, we verify the theoretical results through some numerical examples.

\section{Preliminaries} \label{Sec2}

\subsection{Notation}

We use $|\cdot|$ to denote the trace norm of a matrix or vector. The scalar product of two vectors is denoted by $\LL\cdot, \cdot\RR$. Let $\mbf L^p(\Omega, \mathcal F, \mbf P;\mbb R^d)$, $p \ge 1$, be the Banach space consisting of $p$th integrable $\mbb R^d$-valued random variables $X$, equipped with the usual norm $\|X\|_{\mbf L^p(\Omega)} := (\mbf E|X|^p)^{1/p}$. Denote by $\mbf P\circ X^{-1}$ the distribution of a random element $X$ defined on the probability space $(\Omega, \mathcal{F}, \mbf P) $. Let $\lfloor \cdot\rfloor$ and $\lceil\cdot\rceil$ stand for the floor function and ceiling function of a real number, respectively. Additionally, $ \overset{d}{\Longrightarrow} $ (resp.\ $\overset{\mbf P}{\longrightarrow}$) denotes the convergence in distribution (resp.\ in probability) of a family of random variables.

Let $\mathbf C(\mathbb R^d)$ (resp.\ $\mathbf C^{k}(\mathbb R^d)$) denote the space of continuous (resp.\ $k$th continuously differentiable) functions defined on $\mathbb R^d$. For a real-valued function $f\in\mathbf C^k(\mathbb R^d)$, $\mathcal D^k f(x) (\xi_1, \ldots, \xi_k)$ denotes the $k$th order G\v ateaux derivative along the directions $\xi_1, \ldots, \xi_k\, \in\mathbb R^d$. 
For an $\mathbb R^m$-valued function $f = (f_1, \ldots, f_m)^\top\in\mathbf C^k(\mathbb R^d)$, we define $\mathcal D ^kf(x) (\xi_1, \ldots, \xi_k)$ as the vector $(\mathcal D ^kf_1(x) (\xi_1, \ldots, \xi_k), \ldots, \mathcal D ^kf_m(x) (\xi_1, \ldots, \xi_k))^\top$.
Let $\mathbf F$ be the set of functions with at most polynomial growth, i.e., a (tensor-valued) function $f\in\mathbf F$ means that there exist constants $C>0$ and $\eta>0$ such that for any $x\in\mathbb R^d$, $|f(x) | \le C(1 + |x|^\eta)$ or $\|f(x) \|_{\otimes} \le C(1 + |x|^\eta)$, where $\otimes$ denotes the norm of a tensor. 
Let $\mbf C^{\alpha}([a, b];\R^d)$, $0<\alpha \le 1$, be the Banach space consisting of all $\alpha$-uniformly H\"older continuous $\R^d$-valued functions defined on $[a, b]\subseteq\mbb R$, equipped with the norm $\|f\|_{\mbf C^{\alpha}([a, b])} := [f]_{\mbf C^{\alpha}([a, b])} + \|f\|_{\mbf C([a, b])}$, where the semi-norm $[f]_{\mbf C^{\alpha}([a, b])} := \sup_{\substack{t, s\in[a, b]\\ t\neq s}}\frac{|f(t)-f(s)|}{|t-s|^\alpha}$ and the supremum norm $\|f\|_{\mbf C([a, b])}:= \sup_{t\in[a, b]}|f(t)|$.
Throughout this paper, let $K(a_1, a_2, ..., a_m)$ be some generic constant which depends on the parameters $a_1, a_2, \ldots, a_m$ but is independent of the temporal step-size, which may vary for each appearance.

\subsection{Stable convergence in distribution}

In this part, we introduce the notion of stable convergence in distribution and present two useful lemmas.

Let $(E, d_E)$ be a Polish space with the metric $d_E$. Consider a sequence $\{X_n\}_{n = 1}^\infty$ of $E$-valued random variables, each defined on the probability space $(\Omega, \mathcal{F}, \mathbf{P})$.
Let $(\widetilde{\Omega}, \widetilde{\mathcal{F}}, \widetilde{\mathbf{P}})$ denote an extension of $(\Omega, \mathcal{F}, \mathbf{P})$ and let $X$ be an $E$-valued random variable defined on the extension.
We say that $X_n$ stably converges in distribution to $X$ in $E$, denoted by $X_n\overset{stably}{\Longrightarrow}X$ in $E$, if $\mathbf{E}[Z f(X_n)]$ converges to $ \widetilde{\mathbf{E}}[Z f(X) ]$ for every bounded and continuous $f:E\to \mathbb{R}$ and every bounded random variable $Z$ on $(\Omega, \mathcal{F})$. Here, $\widetilde{\mathbf E}$ represents the expectation under the measure $\widetilde{\mathbf{P}}$. As an immediate consequence of the above definition, $X_n\overset{stably}{\Longrightarrow}X$ implies $X_n \overset{d}{\Longrightarrow} X$.
For more details on the stable convergence in distribution, we refer the reader to \cite{Jacod97}. Two useful properties about stable convergence in distribution are as follows.

\begin{pro} \cite[Lemma 2.1]{Protter1998AOP} \label{Pro1}
Let $Y$ be a random variable taking values in another Polish space $F$. If $X_n\overset{stably}{\Longrightarrow}X$ in $E$, then $(Y, X_n)\overset{stably}{\Longrightarrow}(Y, X)$ in $F\times E$.
\end{pro}

\begin{pro} \cite[Lemma 2.4]{SDR2025} \label{Pro2}
Let $\{Y_n\}_{n = 1}^\infty$ be a sequence of E-valued random variables. Assume that $X_n\overset{stably}{\Longrightarrow} X$ in E, and the distance $d_{E}(Y_n, X_n)$ converges in probability to $0$. Then $Y_n\overset{stably}{\Longrightarrow}X$ in E.
\end{pro}

\section{Strong convergence of tamed Euler method for multiplicative noise} \label{Sec3}

In this section, we establish the strong convergence of the tamed Euler method for SDEs with multiplicative noise, where the following assumptions are imposed.

	\begin{assum} \label{assum1}
		Assume that the following conditions hold.
		\begin{itemize}
			
			\item [(A-1)] There exist constants $L$, $l>0$ and $p_0>2$ such that
			\begin{align}
				& 2\LL x-y, f(x) -f(y)\RR + (p_0-1)|g(x) -g(y)|^2 \le L|x-y|^2, \label{coupled monotone} \\
				& |f(x) -f(y)| \le L(1 + |x|^l + |y|^l)|x-y|, \label{fmonotone}
			\end{align}
			for all $x$, $y\in \R^d$.
			
			\item [(A-2)] $X_0 \in \mbf L^{p_0}(\Omega, \mathcal F, \mbf P;\mbb R^d)$. 
		\end{itemize}
	\end{assum}

	\begin{rem} \label{Rem1}
		It follows from Assumption \ref{assum1} that there exists $L_1>0$ such that
		\begin{align}
			&|f(x) | \le L_1(1 + |x|^{l + 1}), \label{f l}\\
			&|g(x) |^2 \le L_1(1 + |x|^{l + 2}), \label{g l}\\
			&|g(x) -g(y)|^2 \le L_1(1 + |x|^l + |y|^l)|x-y|^2 \label{g M}, \\
			&2\LL x, f(x) \RR + (p_1-1)|g(x) |^2 \le L_1(1 + {|x|}^2) \label{monotone}, 
		\end{align} 
		for all $x$, $y\in\R^d$ and $p_1<p_0$.
		Then under Assumption \ref{assum1}, SDE \eqref{SDE1} admits a unique strong solution given by
		\begin{align} \label{functionX}
			X(t) = X_0 + \int_{0}^{t}f(X(s)) \ud s + \int_{0}^{t}g(X(s)) \ud \mbf W(s), \qquad t\in[0, T].
		\end{align}
		Moreover, for any $p< p_0$, and $t\in[0, T]$, 
		\begin{align}
			\|X(t)\|_{\mbf L^p(\Omega)}& \le C_1(p)e^{C_1(p)T} \big( 1 + \|X_0\|_{\mbf L^p(\Omega)} \big), \label{Ybound}
		\end{align}
		where $C_1(p)$ depends on $p$ but independent of $T$ (cf.\ \cite[Chapter 2.4]{Maoxuerong}).
	\end{rem}

	Consider the following tamed Euler method for SDE \eqref{SDE1}:
	\begin{align} \label{TEM}
		\bar{X}_{n + 1}^{\alpha} = \bar{X}_n^{\alpha} + \frac{T}{N}f^{\alpha}(\bar{X}_n^{\alpha}) + g^{\alpha}(\bar{X}_n^{\alpha}) \left( \mbf W\big(\frac{(n + 1)T}{N}\big)-\mbf W\big(\frac{nT}{N}\big) \right), 
	\end{align}
	for $n = 0, 1, \ldots, N-1$, $N\in \mathbb N$ with $\bar{X}_0^{\alpha} = X_0$, where $f^{\alpha}(x) := \frac{f(x) }{1 + (\frac{T}{N})^{\alpha}|x|^{2l}}$ and $g^{\alpha}(x) := \frac{g(x) }{1 + (\frac{T}{N})^{\alpha}|x|^{2l}}$ with $\alpha\in (0, 1]$.

For convenience, we introduce the continuous version of ${\lbrace \bar {X}_n^{\alpha}\rbrace}_{n = 0}^N$:
	\begin{align} \label{CTEM}
		X^\alpha_N(t) = X_0 + \int_{0}^{t} f^{\alpha}(X^\alpha_N(\kappa_N(s))) \ud s + \int_{0}^{t} g^{\alpha}(X^\alpha_N(\kappa_N(s))) \ud \mbf W(s), \qquad t\in(0, T], 
	\end{align}
where $\kappa_N(s) = \lfloor\frac{Ns}{T}\rfloor\frac{T}{N}$, $s\in[0, T]$. It is easy to see that $X_N^\alpha(t_k) = \bar{X}^\alpha_k$, $k = 0, 1, \ldots, N$.

To investigate the strong convergence order of the tamed Euler method \eqref{CTEM}, we first establish the following lemma on the $\mbf L^p$-boundedness of the numerical solution.
\begin{lem} \label{TEMbound}
		Let Assumption \ref{assum1} hold. Then for every positive $p< p_0$, 
		\begin{align*}
			\sup_{N\in \mathbb{N}}\sup_{t\in[0, T]}\mathbf{E}|X^\alpha_N(t)|^p \le K(p, T, \mbf E|X_0|^{p_0}).
		\end{align*}
\end{lem}

\begin{proof}
Based on \eqref{monotone}, a direct calculation shows that for any $p< p_0$ and $x\in \R^d$, 
		\begin{align} \label{tamed monotone}
			2\LL x, f^\alpha(x) \RR + (p-1){|g^\alpha(x) |}^2 \le K(1 + {|x|}^2).
		\end{align}
Due to $a + b \geq 2\sqrt{ab}$, for $a$, $b \geq0$, \eqref{f l}, and \eqref{g l}, one obtains that for $x\in \R^d$, 
		\begin{align}
			& |f^{\alpha}(x) | \le \frac{L_1(1 + |x|^{l + 1})}{1 + (\frac{T}{N})^\alpha|x|^{2l}} \le K(T)N^{\frac{\alpha}{2}}(1 + |x|),\\
			& |g^{\alpha}(x) |^2 \le \frac{L_1(1 + |x|^{l + 2})}{1 + 2(\frac{T}{N})^\alpha|x|^{2l}} \le K(T)N^{\frac{\alpha}{2}}(1 + |x|^2) \label{tamed g}, 
		\end{align}
with $\frac{\alpha}{2}\in (0, \frac{1}{2}]$.
By Assumption \ref{assum1}, and \eqref{tamed monotone}--\eqref{tamed g}, it follows from \cite[Lemma 2]{Sabanis2016} that for any $p<p_0$, $\sup_{N\in \mathbb{N}}\sup_{t\in[0, T]}\mathbf{E}|X^\alpha_N(t)|^p \le K(p, T, \mbf E|X_0|^{p_0})$, which completes the proof.
\end{proof}

Based on Lemma \ref{TEMbound}, we derive the following two lemmas which are useful for the error analysis.

\begin{lem} \label{TEMfunction}
		Let Assumption \ref{assum1} hold. Then for any $0<p< \frac{p_0}{3l + 1}$, 
		\begin{align}
			&\; \mbf E\int_{0}^{T}|f^{\alpha}(X^\alpha_N(\kappa_N(s)))-f(X^\alpha_N(\kappa_N(s)))|^p \ud s \le K(p, T, \mbf E|X_0|^{p_0})N^{-\alpha p}, \\
			&\; \mbf E\int_{0}^{T}|g^{\alpha}(X^\alpha_N(\kappa_N(s)))-g(X^\alpha_N(\kappa_N(s)))|^p \ud s \le K(p, T, \mbf E|X_0|^{p_0})N^{-\alpha p}.
		\end{align}
	\end{lem}

\begin{proof}
The application of Lemma \ref{TEMbound} and \eqref{f l} yields that for $p< \frac{p_0}{3l + 1}$, 
		\begin{align*}
			&\ \mbf E\int_{0}^{T}|f^{\alpha}(X^\alpha_N(\kappa_N(s)))-f(X^\alpha_N(\kappa_N(s)))|^p \ud s\\
			& \le \mbf E\int_{0}^{T}(\frac{T}{N})^{\alpha p}|X^\alpha_N(\kappa_N(s))|^{2lp}|f(X^\alpha_N(\kappa_N(s)))|^p \ud s\\
			& \le L_1^p(\frac{T}{N})^{\alpha p}\int_{0}^{T}\Big (\mbf E |X^\alpha_N(\kappa_N(s))|^{2lp} + \mbf E |X^\alpha_N(\kappa_N(s))|^{(3l + 1)p}\Big ) \ud s\\
			& \le K(p, T, \mbf E|X_0|^{p_0})N^{-\alpha p}.
		\end{align*} 
Similarly, one can show that for $p< \frac{p_0}{\frac{5}{2}l + 1}$, $\mbf E\int_{0}^{T}|g^{\alpha}(X^\alpha_N(\kappa_N(s)))-g(X^\alpha_N(\kappa_N(s)))|^p \ud s \le K(p, T, \mbf E|X_0|^{p_0})N^{-\alpha p}$, which completes the proof.
\end{proof}

\begin{lem} \label{TEMnumberical}
		Let Assumption \ref{assum1} hold. Then for any positive $p< \frac{p_0}{l + 1}$, and $t$, $s\in[0, T]$, 
		\begin{align*}
			\mbf E|X^\alpha_N(t)-X^\alpha_N(s)|^p \le K(p, T, \mbf E|X_0|^{p_0})|t-s|^{\frac{p}{2}}.
		\end{align*}
\end{lem}

\begin{proof}
Using \eqref{CTEM} and H\"older's inequality shows 
		\begin{align*}
			\mbf E|X^\alpha_N(t)-X^\alpha_N(s)|^p \le 2^{p-1}\mbf E\Big|\int_{s}^{t}f^{\alpha}(X^\alpha_N(\kappa_N(r))) \ud r\Big|^p + 2^{p-1}\mbf E\Big|\int_{s}^{t}g^{\alpha}(X^\alpha_N(\kappa_N(r))) \ud \mbf W(r)\Big|^p.
		\end{align*}
It follows from Lemma \ref{TEMbound}, \eqref{f l} and H\"older's inequality that for $p< \frac{p_0}{l + 1}$, 
		\begin{align*}
			\mbf E\Big|\int_{s}^{t}f^{\alpha}(X^\alpha_N(\kappa_N(r))) \ud r\Big|^p \le K(p, T, \mbf E|X_0|^{p_0})|t-s|^{p} \le K(p, T, \mbf E|X_0|^{p_0})|t-s|^{\frac{p}{2}}.
		\end{align*}
Applying H\"older's inequality and the BDG inequality, we deduce from \eqref{g l} and Lemma \ref{TEMbound} that for $p< \frac{p_0}{\frac{l}{2} + 1}$, 
		\begin{align*}
			\mbf E\Big|\int_{s}^{t}g^{\alpha}(X^\alpha_N(\kappa_N(r))) \ud \mbf W(r)\Big|^p \le K(p, T, \mbf E|X_0|^{p_0})|t-s|^{\frac{p}{2}}.
		\end{align*}
		This completes the proof.
\end{proof}

	The following corollary is very useful in the estimation for the remainder terms in the expansion of the normalized error process.
	
	\begin{cor} \label{TEMuniformbound}
		Let Assumption \ref{assum1} hold with $l< \frac{p_0-2}{2}$. Then for any $0<p< \frac{p_0}{l + 1}$, 
		\begin{align*}
			\mbf E\|X^\alpha_N\|^p_{\mathbf{C}([0, T])} \le K(p, T, \mbf E|X_0|^{p_0}).
		\end{align*}
	\end{cor}

\begin{proof}
Since $l< \frac{p_0-2}{2}$, it follows from Lemma \ref{TEMnumberical} and Kolmogorov's continuity theorem (cf.\ \cite[Appendix C.6]{KolomoContinuous}) that for any $p\in(2 , \frac{p_0}{l + 1})$, there exists $q\in(0, 1-\frac{2}{p})$ such that
		\begin{align*}
			\mbf E \bigg( \sup_{\substack{s, t\in[0, T]\\s\neq t}}\frac{|{X}_N^\alpha(t)-{X}_N^\alpha(s)|^p}{|t-s|^{\frac{pq}{2}}} \bigg) 
			 \le K(p, T, \mbf E|X_0|^{p_0}).
		\end{align*}
By taking $s = 0$, we can derive that for any $p\in(2 , \frac{p_0}{l + 1})$, $\mbf E (\sup_{t \in(0, T]}|X_N^{\alpha}(t)-X_0|^p) \le K(p, T, \mbf E|X_0|^{p_0}) $. Thus, the proof is complete by applying H\"older's inequality.
\end{proof}

With previous preparation, we now can give the strong convergence rate of the tamed Euler method \eqref{CTEM}.

\begin{theo} \label{TEMconverge}
Let Assumption \ref{assum1} hold with $l< \frac{p_0-2}{6}$. Then for any $0<p< \frac{p_0}{3l + 1}$, the solution of  \eqref{CTEM} converges to the solution of SDE \eqref{SDE1} in $\mbf{L}^p$-sense with order $\frac{1}{2}\wedge\alpha$, that is, 
\begin{align} \label{strong convergence}
\sup_{t\in[0, T]}\mbf E|X^\alpha_N(t)-X(t)|^p \le K(p, T, \mbf E|X_0|^{p_0})N^{-(\alpha\wedge\frac{1}{2})p }.
\end{align}
\end{theo}

\begin{proof}
		For $2 \le p<p_0$, It\^o's formula gives that 
		\begin{align*}
			\mbf E|X(t)-X^\alpha_N(t)|^p
			& \le \frac{p}{2}\mbf E\int_{0}^{t}|X(s)-X^\alpha_N(s)|^{p-2}\big (2\LL X(s)-X^\alpha_N(s), f(X(s))-f^\alpha (X^\alpha_N(\kappa_N(s)))\RR \\
			&\quad + (p-1)|g(X(s))-g^\alpha(X^\alpha_N(\kappa_N(s)))|^2\big ) \ud s = :\frac{p}{2}\mbf E\int_{0}^{t}\mathcal{E}(s) \ud s.
		\end{align*} 
		Using the inequality $|a + b|^2 \le (1 + \epsilon)|a|^2 + (1 + \frac{1}{\epsilon})|b|^2$, for any $a$, $b\in \R^d$ and $\epsilon>0$, a direct calculation shows that for $\epsilon_0>0$ with $(1 + \epsilon_0)(p-1) \le (p_0-1)$, 
		\begin{align*}
			&\ 2\LL X(s)-X^\alpha_N(s), f(X(s))-f^\alpha (X^\alpha_N(\kappa_N(s)))\RR + (p-1)|g(X(s))-g^\alpha(X^\alpha_N(\kappa_N(s)))|^2\\
			& \le 2\LL X(s)-X^\alpha_N(s), f(X(s))-f(X^\alpha_N(s))\RR + (1 + \epsilon_0)(p-1)|g(X(s))-g(X^\alpha_N(s))|^2\\
			&\quad + 2\LL X(s)-X^\alpha_N(s), f(X^\alpha_N(s))-f(X^\alpha_N(\kappa_N(s)))\RR \\
			&\quad + 2(1 + \frac{1}{\epsilon_0})(p-1)|g(X^\alpha_N(s))-g(X^\alpha_N(\kappa_N(s)))|^2 \\
			&\quad + 2\LL X(s)-X^\alpha_N(s), f(X^\alpha_N(\kappa_N(s)))-f^\alpha(X^\alpha_N(\kappa_N(s)))\RR \\
			&\quad + 2(1 + \frac{1}{\epsilon_0})(p-1)|g(X^\alpha_N(\kappa_N(s)))-g^\alpha(X^\alpha_N(\kappa_N(s)))|^2.
		\end{align*}
		Combining Assumption \ref{assum1}, \eqref{g M}, and using Young's inequality, we have that
		\begin{align*}
			\mathcal{E}(s)
			& \le (L + 2)|X(s)-X^\alpha_N(s)|^p + K\big (1 + |X^\alpha_N(s)|^{l} + |X^\alpha_N(\kappa_N(s))|^{l}\big )^{p}|X^\alpha_N(s)-X^\alpha_N(\kappa_N(s))|^p\\
			&\quad + K|f^{\alpha}(X^\alpha_N(\kappa_N(s)))-f(X^\alpha_N(\kappa_N(s)))|^p + K|g^{\alpha}(X^\alpha_N(\kappa_N(s)))-g(X^\alpha_N(\kappa_N(s)))|^p\\
			& = : (L + 2)|X(s)-X^\alpha_N(s)|^p + e(s), \quad s\in[0, T].
		\end{align*}
		Through H\"older's inequality, and Lemmas \ref{TEMbound} and \ref{TEMnumberical}, it follows that for $p< \frac{p_0}{3l + 1}$, 
		\begin{align} \label{est}
			&\ \mbf E\big ((1 + |X^\alpha_N(s)|^{l} + |X^\alpha_N(\kappa_N(s))|^{l})^{p}|X^\alpha_N(s)-X^\alpha_N(\kappa_N(s))|^p\big )\notag\\
			& \le \Big (\mbf E\big (1 + |X^\alpha_N(s)|^{l} + |X^\alpha_N(\kappa_N(s))|^{l}\big )^{p\times \frac{3l + 1}{l}}\Big )^{\frac{l}{3l + 1}} 
				\times \Big (\mbf E|X^\alpha_N(s)-X^\alpha_N(\kappa_N(s))|^{p\times \frac{3l + 1}{2l + 1}}\Big )^{\frac{2l + 1}{3l + 1}}\notag\\
			& \le K(p, T, \mbf E|X_0|^{p_0})N^{-\frac{p}{2}}, 
		\end{align}
		due to $pl\times \frac{3l + 1}{l}< p_0$ and $p\times \frac{3l + 1}{2l + 1}< \frac{p_0}{l + 1}$ for $p< \frac{p_0}{3l + 1} $. 
		
		Then Lemma \ref{TEMfunction} and \eqref{est} yield that
		\begin{align*}
			\mbf E\int_{0}^{t}e(s) \ud s \le K(p, T, \mbf E|X_0|^{p_0})N^{-p(\frac{1}{2}\wedge\alpha)}, \qquad t\in [0, T].
		\end{align*}
		This implies that $\mbf E|X(t)-X^\alpha_N(t)|^p \le K\int_{0}^{t}\mbf E|X(s)-X^\alpha_N(s)|^p \ud s + K(p, T, \mbf E|X_0|^{p_0})N^{-p(\frac{1}{2}\wedge\alpha)} $. Then it follows from Gr\"onwall's inequality that $\mbf E|X(t)-X^\alpha_N(t)|^p \le K(p, T, \mbf E|X_0|^{p_0})N^{-p(\frac{1}{2}\wedge\alpha)}$, $t\in [0, T] $. For $0<p<2$, the application of H\"older's inequality yields \eqref{strong convergence}. The proof is complete. 
\end{proof}

	The following theorem indicates that $\alpha $ can take arbitrary positive values by adjusting the regularization coefficient without altering the strong convergence rate.

	\begin{theo} \label{M2}
		Let Assumption \ref{assum1} hold and $\alpha>0$. Consider the numerical scheme \eqref{CTEM} with $f^{\alpha}(x) = \frac{f(x) }{1 + (\frac{T}{N})^{\alpha}|x|^{l'}}$, $g^{\alpha}(x) = \frac{g(x) }{1 + (\frac{T}{N})^{\alpha}|x|^{l'}}$, where $l'>0$ and $x\in \R^d $. If $l' \geq \lceil 2\alpha\rceil l$ and $l<\frac{p_0}{2}-l'-1$, then for any $p<\frac{p_0}{l' + l + 1}$, the corresponding numerical solution converges to the solution of SDE \eqref{SDE1} in $\mbf {L}^p$-sense with order $\frac{1}{2}\wedge\alpha $. 
	\end{theo}

	\begin{proof}
		Applying the inequality $\sum_{k = 1}^{n}a_k \geq n\sqrt[n]{a_1\ldots a_n}$, $a_1, \ldots, a_n \geq 0$, we obtain that for $l' \geq \lceil 2\alpha\rceil l$, $|f^{\alpha}(x) | \le K(T)N^{\frac{\alpha}{ \lceil 2\alpha\rceil }}(1 + |x|)$ and $|g^{\alpha}(x) |^2 \le K(T)N^{\frac{\alpha}{ \lceil 2\alpha\rceil }}(1 + |x|^2) $. Due to $\frac{\alpha}{ \lceil 2\alpha\rceil }\in (0, \frac{1}{2}]$, then Lemma \ref{TEMbound} still holds for all $l' \geq \lceil 2\alpha\rceil l $. Following the proof steps of Theorem \ref{TEMconverge}, we can establish this theorem.
	\end{proof}

\begin{rem}
	Comparing with the tamed Euler methods of \cite{Hutzen2012,Sabanis13,Sabanis2016},  the proposed one in Theorem \ref{M2}, with $f^{\alpha}(x) = \frac{f(x) }{1 + (\frac{T}{N})^{\alpha}|x|^{l'}}$ and $g^{\alpha}(x) = \frac{g(x) }{1 + (\frac{T}{N})^{\alpha}|x|^{l'}}$, allows the regularization parameter $\alpha$ to take arbitrary positive numbers, and retains simultaneously a tractable framework for its convergence analysis. 
\end{rem}

	\section{Asymptotic error distribution for multiplicative noise} \label{Sec4}
	
	In this section, we give the asymptotic error distribution of 
	$\bar{X}_N^{\alpha}$. We would like to point out that the analysis in this section  extends easily to the tamed Euler method in Theorem \ref{M2}.

We further pose the following assumption.

	\begin{assum} \label{assum2}
		Assume that $f$, $g_k\in\mbf C^2(\mbb R^d)$ and $\mathcal D^2f, \, \mathcal D^2g_k\in\mbf F$, $k = 1, 2, \ldots, m$. Moreover, assume that $p_0$ in \eqref{coupled monotone} is sufficiently large.
	\end{assum}

Next,	we  give an expansion for the normalized error $N^{\frac{1}{2}\wedge\alpha}( X^\alpha_N(t)-X(t))$. 
	
	\begin{lem} \label{UNexpression}
		Denote $U^{\alpha, N}(t) := N^{\alpha\wedge\frac{1}{2}}( X^\alpha_N(t)-X(t))$, $t\in[0, T]$. Let Assumptions \ref{assum1} and \ref{assum2} hold. Then $U^{\alpha, N}$ has the following representation
		\begin{align} \label{fgexan}
			U^{\alpha, N}(t)
			& = \int_0^t\nabla f(X(s))U^{\alpha, N}(s) \ud s + \sum_{k = 1}^{m}\int_0^t\nabla g_k(X(s))U^{\alpha, N}(s) \ud W_k(s)\notag \\
			&\quad + \sum_{i = 0}^{m + 1}I_i^{\alpha, N}(t) + R^{\alpha, N}(t), \qquad t\in[0, T].
		\end{align}	
		Here, 
		\begin{align}
			& I_0^{\alpha, N}(t) := -T^{\alpha}N^{(\alpha\wedge\frac{1}{2})-\alpha} \int_0^t\frac{f(X^\alpha_N(\kappa_N(s)))|X^\alpha_N(\kappa_N(s))|^{2l}}{1 + (\frac{T}{N})^{\alpha}|X^\alpha_N(\kappa_N(s))|^{2l}} \ud s, \label{I0ex}\\
			& I_i^{\alpha, N}(t) := -T^{\alpha}N^{(\alpha\wedge\frac{1}{2})-\alpha}\int_0^t\frac{g_i(X^\alpha_N(\kappa_N(s)))|X^\alpha_N(\kappa_N(s))|^{2l}}{1 + (\frac{T}{N})^{\alpha}|X^\alpha_N(\kappa_N(s))|^{2l}} \ud W_i(s), \qquad i = 1, \ldots, m, \label{Ik} \\
			& I_{m + 1}^{\alpha, N}(t) := -N^{\alpha\wedge\frac{1}{2}}\sum_{k = 1}^{m}\sum_{u = 1}^{m}\int_0^t\nabla g_k(X^\alpha_N(\kappa_N(s)))g_u^{\alpha}(X^\alpha_N(\kappa_N(s))\int_{\kappa_N(s)}^{s} \ud W_u(r) \ud W_k(s),
		\end{align}
		and $R^{\alpha, N}(t)$ is the remainder with $\lim_{N\to \infty}\mbf E\| R^{\alpha, N}\|^2_{\mathbf{C}([0, T])} = 0$.
\end{lem}

	\begin{proof}
		Combining \eqref{functionX} and \eqref{CTEM}, we have
		\begin{align} \label{U0}
			U^{\alpha, N}(t)
			& = N^{\alpha\wedge\frac{1}{2}}\int_0^t \big [f^\alpha(X^\alpha_N(\kappa_N(s)))-f(X(s))\big ] \ud s\notag \\
			&\quad + N^{\alpha\wedge\frac{1}{2}}\sum_{k = 1}^{m}\int_0^t \big [g_k^\alpha(X^\alpha_N(\kappa_N(s)))-g_k(X(s))\big ] \ud W_k(s), 
		\end{align}
		where $g_k^\alpha(x) = \frac{g_k(x) }{1 + (\frac{T}{N})^\alpha|x|^{2l}}$, $x\in\R^d$, $k = 1, 2\ldots, m $.

		Taylor's formula gives 
		\begin{align} \label{fexpansion}
			&\ N^{\alpha\wedge\frac{1}{2}}\int_0^t \big [f^\alpha(X^\alpha_N(\kappa_N(s)))-f(X(s))\big ] \ud s \notag\\
			& = N^{\alpha\wedge\frac{1}{2}}\int_{0}^{t} \big [f(X^\alpha_N(s))-f(X(s))\big ] \ud s + N^{\alpha\wedge\frac{1}{2}}\int_{0}^{t}\big [f(X^\alpha_N(\kappa_N(s)))-f(X^\alpha_N(s))\big ] \ud s \notag\\ 
			&\quad + N^{\alpha\wedge\frac{1}{2}}\int_{0}^{t}\big [f^\alpha(X^\alpha_N(\kappa_N(s)))-f(X^\alpha_N(\kappa_N(s)))\big ] \ud s\notag \\
			& = \int_0^t\nabla f(X(s))U^{\alpha, N}(s) \ud s + N^{\alpha\wedge\frac{1}{2}}\int_{0}^{t}\nabla f(X^\alpha_N(\kappa_N(s)))(X^\alpha_N(\kappa_N(s))-X^\alpha_N(s)) \ud s\notag\\
			&\quad + R_1^{\alpha, N}(t) + I_0^{\alpha, N}(t), 
		\end{align}
		where $I_0^{\alpha, N}(t)$ is defined by \eqref{I0ex} and 
		\begin{align}
			&R_1^{\alpha, N}(t) := N^{\alpha\wedge\frac{1}{2}}\int_{0}^{t}\int_{0}^{1}\Big[(1-\lambda)\mathcal{D}^2f(X(s) + \lambda(X^\alpha_N(s)-X(s)))(X^\alpha_N(s)-X(s), X^\alpha_N(s)-X(s))\notag\\
			&\;\;\;\;\;\;\;\;\;\;\;\;\;\;\;\;\;\;\;\;\;\;\;\;\;\;\;\;\;\;\;\;\;\;\;\;-(1-\lambda)\mathcal{D}^2 f\big(X^\alpha_N(\kappa_N(s)) + \lambda(X^\alpha_N(s)-X^\alpha_N(\kappa_N(s)))\big)\notag\\ &\;\;\;\;\;\;\;\;\;\;\;\;\;\;\;\;\;\;\;\;\;\;\;\;\;\;\;\;\;\;\;\;\;\;\;\;\;\;\;\;\big(X^\alpha_N(\kappa_N(s))-X^\alpha_N(s), X^\alpha_N(\kappa_N(s))-X^\alpha_N(s)\big )\Big] \ud \lambda \ud s. 
		\end{align}
		Further, it follows from\eqref{CTEM} that 
		$$ X^\alpha_N(s)-X^\alpha_N(\kappa_N(s)) = \int_{\kappa_N(s)}^{s} f^{\alpha}(X^\alpha_N(\kappa_N(r))) \ud r + \sum_{k = 1}^{m} \int_{\kappa_N(s)}^{s} g_k^{\alpha}(X^\alpha_N(\kappa_N(r))) \ud W_k(r). $$
		Therefore we have that
		\begin{align*}
			&\ N^{\alpha\wedge\frac{1}{2}}\int_{0}^{t}\nabla f(X^\alpha_N(\kappa_N(s)))(X^\alpha_N(\kappa_N(s))-X^\alpha_N(s)) \ud s\\
			& = R_2^{\alpha, N}(t)-N^{\alpha\wedge\frac{1}{2}}\sum_{k = 1}^{m}\int_{0}^{t} \nabla f(X^\alpha_N(\kappa_N(s)))g_k^\alpha(X^\alpha_N(\kappa_N(s)))\int_{\kappa_N(s)}^{s} \ud W_k(r) \ud s, 
		\end{align*}
		where 
		\begin{align}
			R_2^{\alpha, N}(t) := N^{\alpha\wedge\frac{1}{2}}\int_{0}^{t}\nabla f(X^\alpha_N(\kappa_N(s)))f^\alpha(X^\alpha_N(\kappa_N(s)))(\kappa_N(s)-s) \ud s.
		\end{align}
		Then we derive from stochastic Fubini theorem that
		\begin{align} \label{stochatic fubini}
			& - N^{\alpha\wedge\frac{1}{2}}\int_{0}^{t} \nabla f(X^\alpha_N(\kappa_N(s)))g_k^\alpha(X^\alpha_N(\kappa_N(s)))\int_{\kappa_N(s)}^{s} \ud W_k(r) \ud s\notag\\
			& = -N^{\alpha\wedge\frac{1}{2}}\int_{0}^{t}\nabla f(X^\alpha_N(\kappa_N(r)))g_k^\alpha(X^\alpha_N(\kappa_N(r)))(\kappa_N(r) + \frac{T}{N}-r) \ud W_k(r)\notag\\
			&\quad - N^{\alpha\wedge\frac{1}{2}}\int_{0}^{t}\nabla f(X^\alpha_N(\kappa_N(r)))g_k^\alpha(X^\alpha_N(\kappa_N(r)))((\kappa_N(r) + \frac{T}{N})\wedge t-\kappa_N(r)-\frac{T}{N}) \ud
			W_k(r)\notag\\
			& = -N^{\alpha\wedge\frac{1}{2}}\int_{0}^{t}\nabla f(X^\alpha_N(\kappa_N(r)))g_k^\alpha(X^\alpha_N(\kappa_N(r)))(\kappa_N(r) + \frac{T}{N}-r) \ud W_k(r)\notag\\
			&\quad -N^{\alpha\wedge\frac{1}{2}}\int_{\kappa_N(t)}^{t}\nabla f(X^\alpha_N(\kappa_N(r)))g_k^\alpha(X^\alpha_N(\kappa_N(r)))( t-\kappa_N(t)-\frac{T}{N}) \ud
			W_k(r), 
		\end{align}
		due to the definition of $\kappa_N $.

		Consequently, we deduce from \eqref{fexpansion}--\eqref{stochatic fubini} that
		\begin{align} \label{fexpan}
			&\ N^{\alpha\wedge\frac{1}{2}}\int_0^t \big [f^\alpha(X^\alpha_N(\kappa_N(s)))-f(X(s))\big ] \ud s \notag\\
			& = \int_0^t\nabla f(X(s))U^{\alpha, N}(s) \ud s + I_0^{\alpha, N}(t) + R_1^{\alpha, N}(t) + R_2^{\alpha, N}(t) + R_3^{\alpha, N}(t) + R_4^{\alpha, N}(t), 
		\end{align}
		where 
		\begin{align}
			&R_3^{\alpha, N}(t) := -N^{\alpha\wedge\frac{1}{2}}\sum_{k = 1}^{m}\int_{0}^{t}\nabla f(X^\alpha_N(\kappa_N(r)))g_k^\alpha(X^\alpha_N(\kappa_N(r)))(\kappa_N(r) + \frac{T}{N}-r) \ud W_k(r),\\
			&R_4^{\alpha, N}(t) := -N^{\alpha\wedge\frac{1}{2}}\sum_{k = 1}^{m}\int_{\kappa_N(t)}^{t}\nabla f(X^\alpha_N(\kappa_N(r)))g_k^\alpha(X^\alpha_N(\kappa_N(r)))( t-\kappa_N(t)-\frac{T}{N}) \ud
			W_k(r).
		\end{align}
		For the stochastic integral in \eqref{U0}, similar to \eqref{fexpansion}, we obtain that for $k = 1, 2, \ldots, m$, 
		\begin{align}{ \label{U1}}
			&\ N^{\alpha\wedge\frac{1}{2}}\int_0^t \big [g_k^\alpha(X^\alpha_N(\kappa_N(s)))-g_k(X(s))\big ] \ud W_k(s) \notag\\
			& = \int_{0}^{t}\nabla g_k(X(s))U^{\alpha, N}(s) \ud W_k(s) + R_{5, k}^{\alpha, N}(t) + I_k^{\alpha, N}(t) \notag\\
			&\quad + N^{\alpha\wedge\frac{1}{2}}\int_{0}^{t}\big [g_k(X^\alpha_N(\kappa_N(s)))-g_k(X^\alpha_N(s))\big ] \ud W_k(s), 
		\end{align}
		where $I_k^{\alpha, N}(t)$ is defined by \eqref{Ik} and 
		\begin{align}
			&R_{5, k}^{\alpha, N}(t) := N^{\alpha\wedge\frac{1}{2}}\int_{0}^{t}\int_{0}^{1}(1-\lambda)\mathcal{D}^2g_k(X(s) + \lambda(X^\alpha_N(s)-X(s)))\notag\\
			&\;\;\;\;\;\;\;\;\;\;\;\;\;\;\;\;\;\;\;\;\;\;\;\;\;\;\;\;\;\;\;\;\;\;\;\;\;\;\;\;\;\;\;\;\;\;\;\;\;\;(X^\alpha_N(s)-X(s), X^\alpha_N(s)-X(s)) \ud \lambda \ud W_k(s). 
		\end{align}
		For the last term in \eqref{U1}, it can be shown that
		\begin{align*}
			&\ N^{\alpha\wedge\frac{1}{2}}\int_{0}^{t}\big [g_k(X^\alpha_N(\kappa_N(s)))-g_k(X^\alpha_N(s))\big ] \ud W_k(s)\\
			& = -N^{\alpha\wedge\frac{1}{2}}\sum_{u = 1}^{m}\int_{0}^{t}\nabla g_k(X^\alpha_N(\kappa_N(s)))g_u^\alpha(X^\alpha_N(\kappa_N(s)))\int_{\kappa_N(s)}^{s} \ud W_u(r) \ud W_k(s) + R_{6, k}^{\alpha, N}(t), 
		\end{align*}
		with
		\begin{align}
			R_{6, k}^{\alpha, N}(t)
			& := N^{\alpha\wedge\frac{1}{2}}\int_{0}^{t}\nabla g_k(X^\alpha_N(\kappa_N(s)))f^\alpha(X^\alpha_N(\kappa_N(s)))(\kappa_N(s)-s) \ud W_k(s)\notag\\
			&\quad\ -N^{\alpha\wedge\frac{1}{2}}\int_{0}^{t}\int_{0}^{1}(1-\lambda)\mathcal{D}^2 g_k(X^\alpha_N(\kappa_N(s)) + \lambda(X^\alpha_N(s)-X^\alpha_N(\kappa_N(s)))\notag\\
			&\qquad\qquad\qquad\qquad \big (X^\alpha_N(\kappa_N(s))-X^\alpha_N(s), X^\alpha_N(\kappa_N(s))-X^\alpha_N(s)\big ) \ud \lambda \ud W_k(s).
		\end{align}
		Thus, \eqref{U1} becomes
		\begin{align} \label{gkexpan}
			&\ N^{\alpha\wedge\frac{1}{2}}\int_0^t \big [g_k^\alpha(X^\alpha_N(\kappa_N(s)))-g_k(X(s))\big ] \ud W_k(s) \notag\\
			& = \int_{0}^{t}\nabla g_k(X(s))U^{\alpha, N}(s) \ud W_k(s) + I_k^{\alpha, N}(t) + R_{5, k}^{\alpha, N}(t) + R_{6, k}^{\alpha, N}(t) \notag\\
			&\quad -N^{\alpha\wedge\frac{1}{2}}\sum_{u = 1}^{m}\int_{0}^{t}\nabla g_k(X^\alpha_N(\kappa_N(s)))g_u^\alpha(X^\alpha_N(\kappa_N(s)))\int_{\kappa_N(s)}^{s} \ud W_u(r) \ud W_k(s).
		\end{align}

		Inserting \eqref{fexpan} and \eqref{gkexpan} into \eqref{U0} yields \eqref{fgexan} with 
		\begin{align}
			R^{\alpha, N}(t) = \sum_{i = 1}^{4}R_i^{\alpha, N}(t) + \sum_{i = 1}^{m}R_{5, i}^{\alpha, N}(t) + \sum_{i = 1}^{m}R_{6, i}^{\alpha, N}(t). \label{R}
		\end{align}
		Then we will show that $\lim_{N\to \infty}\mbf E\| R^{\alpha, N}\|^2_{\mathbf{C}([0, T])} = 0$.
		Applying H\"older's inequality and the BDG inequality, we deduce from \eqref{Ybound}, Lemma \ref{TEMbound}, Lemma \ref{TEMnumberical}, and Theorem \ref{TEMconverge} that 
		\begin{align} 
			& \mbf E\| R_1^{\alpha, N}\|^2_{\mathbf{C}([0, T])} + \mbf E\| R_2^{\alpha, N}\|^2_{\mathbf{C}([0, T])} + \mbf E\| R_3^{\alpha, N}\|^2_{\mathbf{C}([0, T])}\notag \\
			&\qquad + \sum_{k = 1}^{m}\mbf E\| R_{6, k}^{\alpha, N}\|^2_{\mathbf{C}([0, T])} + \sum_{k = 1}^{m}\mbf E\| R_{5, k}^{\alpha, N}\|^2_{\mathbf{C}([0, T])} \le KN^{-(2\alpha\wedge1)}, 
		\end{align}
		due to $\nabla f$, $\nabla g_k\in \mathbf{F}$, $k = 1, \ldots, m $. 
		Recall that 
		$$ R_4^{\alpha, N}(t) = -N^{\alpha\wedge\frac{1}{2}}\sum_{k = 1}^{m}\int_{\kappa_N(t)}^{t}\nabla f(X^\alpha_N(\kappa_N(r)))g_k^\alpha(X^\alpha_N(\kappa_N(r)))( t-\kappa_N(t)-\frac{T}{N}) \ud W_k(r). $$
		By H\"older's inequality, we derive from Corollary \ref{TEMuniformbound} and \eqref{g l} that
		\begin{align*}
			&\ \mbf E\sup_{t \in[0, T]}\Big |\int_{\kappa_N(t)}^{t}\nabla f(X^\alpha_N(\kappa_N(r)))g_k^\alpha(X^\alpha_N(\kappa_N(r)))( t-\kappa_N(t)-\frac{T}{N}) \ud
			W_k(r)\Big |^2\\
			& \le KN^{-2}\Big(\mbf E\sup_{t \in[0, T]}\big |\nabla f(X^\alpha_N(\kappa_N(t)))g_k^\alpha(X^\alpha_N(\kappa_N(t)))\big|^4\Big)^\frac{1}{2}\Big(\mbf E\sup_{t \in[0, T]}\big |W_k(t)-W_k(\kappa_N(t))\big|^4\Big)^\frac{1}{2}\\
			& \le KN^{-2}\Big(\mbf E\sup_{t \in[0, T]}\big |W_k(t)-W_k(\kappa_N(t))\big|^4\Big)^\frac{1}{2}.
		\end{align*}
	 Then it follows from Kolmogorov's continuity theorem (cf.\ \cite[Appendix C.6]{KolomoContinuous}) that there exists $q\in(0, \frac{1}{2})$ such that 
		\begin{align*}
			\mbf E \bigg( \sup_{\substack{s, t\in[0, T]\\s\neq t}}\frac{|W_k(t)-W_k(s)|^4}{|t-s|^{2q}} \bigg) \le K.
		\end{align*}
		Therefore we obtain that 
		\begin{align*}
			\mbf E\sup_{\substack{t\in[0, T]\\ \kappa_N(t)\neq t}}|W_k(t)-W_k(\kappa_N(t))|^4
			& \le\sup_{\substack{t\in[0, T]\\ \kappa_N(t)\neq t}}|t-\kappa_N(t)|^{2q} \, \mbf E \bigg( \sup_{\substack{t\in[0, T]\\ \kappa_N(t)\neq t}}\frac{|W_k(t)-W_k(\kappa_N(t))|^4}{|t-\kappa_N(t)|^{2q}} \bigg)\\
			& \le KN^{-2q}.
		\end{align*}
		By taking $q = \frac{1}{4}$, the application of H\"older's inequality yields 
		\begin{align} \label{R2}
			\mbf E\| R_4^{\alpha, N}\|^2_{\mathbf{C}([0, T])} \le KN^{(2\alpha\wedge1)-\frac{9}{4}}.
		\end{align}
		It follows from \eqref{R}--\eqref{R2} and H\"older's inequality that $\lim_{N\to \infty}\mbf E\| R^{\alpha, N}\|^2_{\mathbf{C}([0, T])} = 0$. Thus, the proof is complete.
	\end{proof}

	The following lemma establishes the convergence in $\mbf L^p$-sense of $I_k^{\alpha, N}$ in $\mbf C([0, T];\R^d)$, $k = 0, \ldots, m $. 
	
	\begin{lem} \label{UN-UtildeN}
		Let Assumptions \ref{assum1} and \ref{assum2} hold. Then for any $\alpha\in (0, 1]$, $\lim_{N\to \infty}\mbf E\|I_k^{\alpha, N}-I_k^{\alpha}\|^2_{\mbf C([0, T])} = 0$, $k = 0, 1, \ldots, m$, 
		in which
		\begin{align*}
			&\;	I_0^{\alpha}(t) = \begin{cases}
				-T^{\alpha}\int_0^tf(X(s))|X(s)|^{2l} \ud s, &\alpha\in (0, \frac{1}{2}], \\
				0, &\alpha\in (\frac{1}{2}, 1], 
			\end{cases}\\
			&\;I_k^{\alpha}(t) = \begin{cases}
				-T^{\alpha}\int_0^tg_k(X(s))|X(s)|^{2l} \ud W_k(s), &\alpha\in (0, \frac{1}{2}], \\
				0, &\alpha\in (\frac{1}{2}, 1], 
			\end{cases}
			\qquad k = 1, 2, \ldots, m.
		\end{align*}
	\end{lem}

	\begin{proof}
		Denote
		\begin{align*}
			\phi_k^{\alpha, N}(t) := \begin{cases}
				-T^\alpha\Big [\frac{g_k(X^\alpha_N(\kappa_N(t)))|X^\alpha_N(\kappa_N(t))|^{2l}}{1 + (\frac{T}{N})^{\alpha}|X^\alpha_N(\kappa_N(t))|^{2l}}-g_k(X(t))|X(t)|^{2l}\Big ], &\alpha\in (0, \frac{1}{2}], \\
				-T^\alpha N^{\frac{1}{2}-\alpha}\Big [\frac{g_k(X^\alpha_N(\kappa_N(t)))|X^\alpha_N(\kappa_N(t))|^{2l}}{1 + (\frac{T}{N})^{\alpha}|X^\alpha_N(\kappa_N(t))|^{2l}}\Big ], &\alpha\in (\frac{1}{2}, 1], 
			\end{cases}
		\end{align*} 
		for $t\in[0, T]$, $k = 0, 1, \ldots, m$, with $g_0 = f $. 
		For $\alpha\in (\frac{1}{2}, 1]$, based on Lemma \ref{TEMbound}, and \eqref{Ybound}, one can show that for $k = 0, 1, \ldots, m$, 
		\begin{align} \label{phi1}
			\mbf E|\phi_k^{\alpha, N}(t)|^2 \le KN^{1-2\alpha}.
		\end{align}
		If $\alpha\in (0, \frac{1}{2}]$, we write 
		\begin{align*}
			\phi_k^{\alpha, N}(s)
			& = -T^\alpha \Big [g_k(X^\alpha_N(\kappa_N(s)))|X^\alpha_N(\kappa_N(s))|^{2l}-g_k(X(s))|X(s)|^{2l}\Big ]\\
			&\quad + T^\alpha \bigg[ g_k(X^\alpha_N(\kappa_N(s)))|X^\alpha_N(\kappa_N(s))|^{2l}-\frac{g_k(X^\alpha_N(\kappa_N(s)))|X^\alpha_N(\kappa_N(s))|^{2l}}{1 + (\frac{T}{N})^{\alpha}|X^\alpha_N(\kappa_N(s))|^{2l}} \bigg], 
		\end{align*}
		for $k = 0, 1, \ldots, m $. 
		Due to Lemma \ref{TEMbound}, one then obtains that 
		\begin{align*}
			\mbf E \bigg| T^\alpha \bigg( g_k(X^\alpha_N(\kappa_N(s)))|X^\alpha_N(\kappa_N(s))|^{2l}-\frac{g_k(X^\alpha_N(\kappa_N(s)))|X^\alpha_N(\kappa_N(s))|^{2l}}{1 + (\frac{T}{N})^{\alpha}|X^\alpha_N(\kappa_N(s))|^{2l}} \bigg) \bigg|^2 
			 \le KN^{-2\alpha}.
		\end{align*}
		It is observed that 
		\begin{align*}
			&\ g_k(X^\alpha_N(\kappa_N(s)))|X^\alpha_N(\kappa_N(s))|^{2l}-g_k(X(s))|X(s)|^{2l}\\
			& = \big [\big (g_k(X^\alpha_N(\kappa_N(s)))-g_k(X(s))\big )|X^\alpha_N(\kappa_N(s))|^{2l}\big ] + \big [g_k(X(s))\big (|X^\alpha_N(\kappa_N(s))|^{2l}-|X(s)|^{2l}\big )\big ].
		\end{align*}
		For $2l \le 1$, we derive that $\big ||X^\alpha_N(\kappa_N(s))|^{2l}-|X(s)|^{2l}\big | \le |X^\alpha_N(\kappa_N(s))-X(s)|^{2l}$, via the inequality $(a + b)^q \le a^q + b^q$, $a$, $b \geq 0$, $q\in[0, 1] $. 
		For $2l> 1$, Taylor's formula gives that $\big ||X^\alpha_N(\kappa_N(s))|^{2l}-|X(s)|^{2l}\big | \le K\big (|X^\alpha_N(\kappa_N(s))|^{2l-1} + |X(s)|^{2l-1}\big )|X^\alpha_N(\kappa_N(s))-X(s)| $. 
		Then combining Lemma \ref{TEMnumberical} and Theorem \ref{TEMconverge}, we get
		\begin{align*}
			\mbf E|X^\alpha_N(\kappa_N(s))-X(s)|^4 \le K\mbf E|X^\alpha_N(\kappa_N(s))-X^\alpha_N(s)|^4 + K\mbf E|X^\alpha_N(s)-X(s)|^4 \le KN^{-4\alpha}, 
		\end{align*}
		where $\alpha\in (0, \frac{1}{2}] $. Consequently, we deduce from \eqref{g l}, \eqref{g M}, \eqref{Ybound}, Lemma \ref{TEMbound}, and H\"older's inequality that for $k = 0, \ldots, m$, and $\alpha\in(0, \frac{1}{2}]$, 
		\begin{align*}
			\mbf E\big |g_k(X^\alpha_N(\kappa_N(s)))|X^\alpha_N(\kappa_N(s))|^{2l}-g_k(X(s))|X(s)|^{2l}\big |^2 \le KN^{-2(2l\wedge1)\alpha}.
		\end{align*}

		This immediately implies that for $k = 0, 1, \ldots, m$, 
		\begin{align} \label{I0}
			\sup_{t\in[0, T]}\mbf E|\phi_k^{\alpha, N}(t)|^2 \le KN^{-2(2l\wedge1)\alpha}, \qquad \alpha\in (0, \frac{1}{2}].
		\end{align}
		Thus, combining \eqref{phi1} and \eqref{I0}, the proof is complete by applying H\"older's inequality and the BDG inequality.
	\end{proof}

	\begin{lem} \label{IN41}
		Let Assumptions \ref{assum1} and \ref{assum2} hold. Then $\theta^{\alpha, N}\overset{stably}{\Longrightarrow}\theta^{\alpha}$ as $\mbf C([0, T];\mbb R^d)$-valued random variables as $N\to\infty$, where $\theta^{\alpha, N}(t) = \sum_{k = 0}^{m + 1}I_{k}^{\alpha, N}(t) + R^{\alpha, N}(t)$, and $\theta^{\alpha}(t) = \sum_{k = 0}^{m + 1}I_{k}^{\alpha}(t)$, $t\in[0, T] $. Here, $\{I_{m + 1}^\alpha(t), \, t\in[0, T]\}$ is the stochastic process defined by
		\begin{align*}
			I_{m + 1}^{\alpha}(t) = \begin{cases}
				0, &\alpha\in(0, \frac{1}{2}), \\
				\frac{\sqrt{2T}}{2}\sum_{k = 1}^{m}\sum_{u = 1}^{m}\int_{0}^{t}\nabla g_k(X(s))g_u(X(s)) \ud \widetilde{W}_{ku}(s), &\alpha\in \left [\frac{1}{2}, 1\right]\!, 
			\end{cases}
		\end{align*} 
		where $\mbf {\widetilde{W}} = (\widetilde{W}_{11}, \widetilde{W}_{12}, \ldots, \widetilde{W}_{1m}, \ldots, \widetilde{W}_{mm})$ is an $m^2$-dimensional standard Brownian motion independent of $\mbf W$.
	\end{lem}

	\begin{proof}
		The proof is mainly based on \cite[Theorem 4-1]{Jacod97}. Hereafter, denote by $\LL X, Y\RR_t$, $t\in[0, T]$ the cross variation process between the real-valued semi-martingales $\{X(t)\}_{t\in [0, T]}$ and $\{Y(t)\}_{t\in [0, T]} $. Denote $\widetilde{g}_{k, u}(x) = \nabla g_k(x) g_u(x) $, and $\widetilde{g}_{k, u}^\alpha(x) = \nabla g_k(x) g^\alpha_u(x) $, $x\in \R^d$, $k$, $u \in\{1, 2, \ldots, d \}$. Let $I_{m + 1}^{\alpha, N, i}$, $\widetilde{g}_{k, u}^{\alpha, i}$, and $\widetilde{g}_{k, u}^{i}$ denote the $i$th component of $I_{m + 1}^{\alpha, N}$, $\widetilde{g}_{k, u}$, and $\widetilde{g}_{k, u}^\alpha$ respectively, $i$, $k$, $u\in \{1, 2, \ldots, d\}$. 
		Recall that
		\begin{align*}
			I_{m + 1}^{\alpha, N}(t) = -N^{\alpha\wedge\frac{1}{2}}\sum_{k = 1}^{m}\sum_{u = 1}^{m}\int_0^t\widetilde{g}_{k, u}^\alpha(X^\alpha_N(\kappa_N(s)))\int_{\kappa_N(s)}^{s} \ud W_u(r) \ud W_k(s).
		\end{align*}
		It can be shown that
		\begin{align*}
			\LL I_{m + 1}^{\alpha, N, i}, W_j\RR_t = -N^{\alpha\wedge\frac{1}{2}}\sum_{u = 1}^{m}\int_0^t\widetilde{g}_{j, u}^{\alpha, i}(X^\alpha_N(\kappa_N(s)))\int_{\kappa_N(s)}^{s} \ud W_u(r) \ud s.
		\end{align*}
		Similar to \eqref{stochatic fubini}, it follows from the BDG inequality and H\"older's inequality that $\mbf E|\LL I_{m + 1}^{\alpha, N, i}, W_j\RR_t|^2 \le KN^{((2\alpha)\wedge1)-2} $. Thus, for any $t\in [0, T]$, $i\in\{1, 2, \ldots, d\}$, and $j\in\{1, \ldots, m\}$, 
		\begin{align} \label{cross1}
			\LL I_{m + 1}^{\alpha, N, i}, W_j\RR_t \overset{\mbf P}{\longrightarrow} 0 \quad \text{ as }N\to \infty.
		\end{align}
		Next we derive the limit of $\LL I_{m + 1}^{\alpha, N, i}, I_{m + 1}^{\alpha, N, j}\RR_t$ with $i$, $j\in\{1, \ldots, d\}$. A direct computation leads to 
		\begin{align*}
			\LL I_{m + 1}^{\alpha, N, i}, I_{m + 1}^{\alpha, N, j}\RR_t
			& = N^{(2\alpha)\wedge1}\sum_{k, u_1, u_2 = 1}^{m}\int_{0}^{t}\widetilde{g}_{k, u_1}^{\alpha, i}(X^\alpha_N(\kappa_N(s)))
			\widetilde{g}_{k, u_2}^{\alpha, j}(X^\alpha_N(\kappa_N(s)))\\
			&\quad \times \Big (\int_{\kappa_N(s)}^{s} \ud W_{u_1}(r_1)\Big )\Big (\int_{\kappa_N(s)}^{s} \ud W_{u_2}(r_2)\Big ) \ud s\\
			& = : \sum_{k, u_1, u_2 = 1}^{m}P_{u_1, u_2}^k(t).
		\end{align*}
		Then It\^o's formula gives 
		\begin{align*}
			P_{u_1, u_1}^k(t)
			& = N^{(2\alpha)\wedge1}\int_{0}^{t}\widetilde{g}_{k, u_1}^{\alpha, i}(X^\alpha_N(\kappa_N(s)))
			\widetilde{g}_{k, u_1}^{\alpha, j}(X^\alpha_N(\kappa_N(s))) \Big( \int_{\kappa_N(s)}^{s} \ud W_{u_1}(r_1) \Big)^2 \ud s\\
			& = N^{(2\alpha)\wedge1}\int_{0}^{t}\widetilde{g}_{k, u_1}^{\alpha, i}(X^\alpha_N(\kappa_N(s)))
			\widetilde{g}_{k, u_1}^{\alpha, j}(X^\alpha_N(\kappa_N(s)))(s-\kappa_N(s)) \ud s\\
			&\quad + 2N^{(2\alpha)\wedge1}\int_{0}^{t}\widetilde{g}_{k, u_1}^{\alpha, i}(X^\alpha_N(\kappa_N(s)))
			\widetilde{g}_{k, u_1}^{\alpha, j}(X^\alpha_N(\kappa_N(s)))\int_{\kappa_N(s)}^{s}\int_{\kappa_N(s)}^{r_1} \ud W_{u_1}(r_2) \ud W_{u_1}(r_1) \ud s\\
			& = : B_1^{\alpha, N}(t) + B_2^{\alpha, N}(t).
		\end{align*}
		Similar to \eqref{cross1}, it can be shown that 
		$\mbf E|\sum_{k = 1}^{m}\sum_{1 \le u_1<u_2 \le m}P_{u_1, u_2}^k(t)|^2 + \mbf E|B_2^{\alpha, N}(t)|^2 \le KN^{-1}$. Further, using \cite[Proposition 4.2]{HJWY24} yields that for $\alpha\in[\frac{1}{2}, 1]$, $B_1^{\alpha, N}(t)$ converges to $\frac{T}{2}\int_{0}^{t}\widetilde{g}_{k, u_1}^{i}(X(s))
		\widetilde{g}_{k, u_1}^{j}(X(s)) \ud s$ in $\mbf L^2(\Omega, \mathcal{F}, \mbf P;\R^d)$. For $\alpha\in (0, \frac{1}{2})$, through straightforward estimations, we obtain that
		$\mbf E|B_1^{\alpha, N}(t)|^2 \le KN^{4\alpha-2}$. This verifies 
		\begin{align} \label{cross2}
			\LL I_{m + 1}^{\alpha, N, i}, I_{m + 1}^{\alpha, N, j}\RR_t 
			\longrightarrow 
			\begin{cases}
				\frac{T}{2}\sum_{k = 1}^{m}\sum_{u = 1}^{m}\int_{0}^{t}\widetilde{g}_{k, u}^{i}(X(s))\widetilde{g}_{k, u}^{j}(X(s)) \ud s, &\alpha\in [\frac{1}{2}, 1], \\
				0, &\alpha\in (0, \frac{1}{2}), 
			\end{cases}
		\end{align}
		in probability as $N\to \infty$, for any $t\in[0, T]$ and $i$, $j = 1, 2, \ldots, d$.

		Combining \eqref{cross1} and \eqref{cross2}, and using \cite[Theorem 4-1]{Jacod97}, we can derive that $I_{m + 1}^{\alpha, N}\overset{stably}{\Longrightarrow}\Psi^\alpha$ as $\mbf C([0, T];\mbb R^d)$-valued random variables as $N\to\infty$, and 
		\begin{align}
			\LL \Psi ^{\alpha, i}, W_j\RR_t& = 0, \label{Psi1}\\
			\LL \Psi^{\alpha, i}, \Psi ^{\alpha, j}\RR_t& = \begin{cases}
				\frac{T}{2}\sum_{k = 1}^{m}\sum_{u = 1}^{m}\int_{0}^{t}\widetilde{g}_{k, u}^{i}(X(s))\widetilde{g}_{k, u}^{j}(X(s)) \ud s, &\alpha\in [\frac{1}{2}, 1], \\
				0, &\alpha\in (0, \frac{1}{2}).
			\end{cases} \label{Psi2}
		\end{align}
		Further, it follows from \cite[Proposition 1-4]{Jacod97} that $\Psi^{\alpha, i}$, $i = 1, \ldots, d$ can be represented as
		\begin{align*}
			\Psi^{\alpha, i}(t) = \sum_{k = 1}^{m}\int_{0}^{t}\mu_k^{\alpha, i}(s) \ud W_k(s) + \sum_{k = 1}^{m}\sum_{u = 1}^{m}\int_{0}^{t}\nu _{k, u}^{\alpha, i}(s) \ud \widetilde{W}_{ku}(s), \qquad t\in [0, T], 
		\end{align*}
		where $\mbf {\widetilde{W}} = (\widetilde{W}_{11}, \widetilde{W}_{12}, \ldots, \widetilde{W}_{1m}, \ldots, \widetilde{W}_{mm})$ is an $m^2$-dimensional standard Brownian motion and is independent of $\mbf W$. By \eqref{Psi1}, we have $\mu_k^{\alpha, i}(s) = 0$, $i = 1, \ldots, d$, $k = 1, \ldots, m$, which along with \eqref{Psi2} gives 
		\begin{align*}
			\sum_{k = 1}^{m}\sum_{u = 1}^{m}\nu _{k, u}^{\alpha, i}(s)\nu _{k, u}^{\alpha, j}(s) = \begin{cases}
				\frac{T}{2}\sum_{k = 1}^{m}\sum_{u = 1}^{m}\widetilde{g}_{k, u}^{i}(X(s))\widetilde{g}_{k, u}^{j}(X(s)), &\alpha\in [\frac{1}{2}, 1], \\
				0, &\alpha\in (0, \frac{1}{2}).
			\end{cases}
		\end{align*}
		Thus, the proof is complete by letting $I_{m + 1}^\alpha = \Psi^\alpha$ and applying Proposition \ref{Pro2}, Lemma \ref{UNexpression}, and Lemma \ref{UN-UtildeN}.
	\end{proof}
	
	\begin{lem} \cite[Lemma 3.2]{Gylemma} \label{uniform lp}
		Let $\{H_t\}_{t \in [0, T]}$ and $ \{M_t\}_{t \in [0, T]}$ be nonnegative continuous $\{\mathcal{F}_t\}$-adapted processes such that for any constant $c > 0$, 
		\[
		\mbf{E} \left[ H_\tau \mathbbm{1}_{\{M_0 \leq c\}} \right] \leq \mbf{E} \left[ M_\tau \mathbbm{1}_{\{M_0 \leq c\}} \right]
		\]
		for any stopping time $\tau \leq T$. Then, for any stopping time $\tau \leq T$ and $\gamma \in (0, 1)$, 
		\[
		\mbf{E} \left[ \sup_{t \leq \tau} H_t^\gamma \right] \leq \frac{2 - \gamma}{1 - \gamma} \mbf{E} \left[ \sup_{t \leq \tau} M_t^\gamma \right].
		\]
	\end{lem}
	
Since $U^{\alpha, N}$ is not continuous w.r.t.\ the Brownian motion (see \eqref{fgexan}), to which the continuous mapping theorem w.r.t.\ the convergence in  distribution does not apply,    we will apply \cite[Theorem 3.2]{HJWY24} to give its limit distribution. For this end, we 	
	denote $Z^{\alpha, N}(t) := 
	\begin{pmatrix}
		X(t)\\U^{\alpha, N}(t)
	\end{pmatrix}$, $t\in [0, T]$. Then $Z^{\alpha, N}$ is the strong solution of
	\begin{align*}
		Z^{\alpha, N}(t) = \int_{0}^{t}F(Z^{\alpha, N}(s)) \ud s + \sum_{k = 1}^{m}\int_{0}^{t}G_k(Z^{\alpha, N}(s)) \ud W_k(s) + \Theta^{\alpha, N}(t), 
	\end{align*}
	where $F(z) = \begin{pmatrix}
		f(x) \\\nabla f(x) y
	\end{pmatrix}$, 
	$G_k(z) = \begin{pmatrix}
		g_k(x) \\\nabla g_k(x) y
	\end{pmatrix}$, $z = \begin{pmatrix}
		x\\y
	\end{pmatrix}$, $x$, $y\in \R^d$, $k = 1, 2, \ldots, m$ and $\Theta^{\alpha, N}(t) = \begin{pmatrix}
		X_0\\\theta ^{\alpha, N}(t)
	\end{pmatrix}$, $t\in [0, T] $.

Furthermore, we again employ the tamed Euler method to discretize $Z^{\alpha, N}$, i.e., let $Z^{\alpha, N, n}$ be the strong solution of 
	\begin{align} \label{Z2}
		Z^{\alpha, N, n}(t) = \int_{0}^{t}\widetilde{F}(Z^{\alpha, N, n}(\kappa_n(s))) \ud s + \sum_{k = 1}^{m}\int_{0}^{t}\widetilde{G}_k(Z^{\alpha, N, n}(\kappa_n(s))) \ud W_k(s) + \Theta^{\alpha, N}(t), 
	\end{align}
	where $\kappa_n(t) = \lfloor\frac{nt}{T}\rfloor\frac{T}{n}$, $n\in \mathbb{N}$, $\widetilde{F}(z) = \frac{1}{1 + \frac{T}{n}|z|^{2l}}F(z)$, $\widetilde{G}_k(z) = \frac{1}{1 + \frac{T}{n}|z|^{2l}}G_k(z)$, $z\in \R^{2d}$, $k = 1, \ldots, m$, and $Z^{\alpha, N, n}(t) = \begin{pmatrix}
		X^n(t)\\U^{\alpha, N, n}(t)
	\end{pmatrix}$, $t\in [0, T] $.

The following lemma shows that $U^{\alpha, N, n}$ satisfies Condition (A1) of \cite[Theorem 3.2]{HJWY24}.
	
	\begin{lem} \label{A1}
		Let Assumptions \ref{assum1} and \ref{assum2} hold. Then 
		$$ \lim_{n\to \infty}\sup_{N \geq 1}\mbf E \|U^{\alpha, N, n}-U^{\alpha, N}\|_{\mathbf{C}([0, T])} = 0. $$
	\end{lem}

	\begin{proof}
		Based on \eqref{coupled monotone}, Taylor's formula gives 
		\begin{align} \label{genbounded}
			2\LL x, \nabla f(y)x\RR + (p_0-1)\sum_{k = 1}^{m}|\nabla g_k(y)x|^2 \le L|x|^2, \quad \text{ $\forall\, x$, $y\in \R^d$}.
		\end{align}
		By \eqref{fmonotone} and \eqref{g M}, Taylor's formula yields 
		\begin{align} \label{Dg}
			\|\nabla f(x) \|_{\otimes} \le K(1 + |x|^l), \quad \|\nabla g_k(x) \|_{\otimes}^2 \le K(1 + |x|^l), \quad k = 1, \ldots, m.
		\end{align}
		For any $x$, $y\in \R^{d}$, and $z = (x^{\top}, y^{\top})^{\top} $, one then obtains that for any $p_1<p_0$, 
		\begin{align} \label{FGbounded}
			&\ 2\LL z, F(z)\RR + (p_1-1)\sum_{k = 1}^{m}|G_k(z)|^2 \notag \\
			& = 2\LL x, f(x) \RR + (p_1-1)\sum_{k = 1}^{m}|g_k(x) |^2 + 2\LL y, \nabla f(x) y\RR + (p_1-1)\sum_{k = 1}^{m}|\nabla g_k(x) y|^2\notag \\
			& \le (L_1 + L)(1 + |x|^2) \le (L_1 + L)(1 + |z|^2), 
		\end{align} 
		due to \eqref{monotone} and \eqref{genbounded}.
		
		Denote $z_i = (x_i^{\top}, y_i^{\top})^{\top} \in \R^{2d}$, where $x_i$, $y_i\in \R^{d}$, $i = 1, 2 $. Then for any $R>0$, by \eqref{coupled monotone} and $\mathcal{D}^2g_k\in \mathbf{F}$, $k = 1, \ldots, m$, we get that there is $K_R>0$ such that 
		\begin{align}
			&\ 2\LL z_1-z_2, F(z_1)-F(z_2)\RR + (p_0-1)\sum_{k = 1}^{m}|G_k(z_1)-G_k(z_2)|^2\notag\\
			& = 2\LL x_1-x_2, f(x_1)-f(x_2)\RR + (p_0-1)\sum_{k = 1}^{m}|g_k(x_1)-g_k(x_2)|^2\notag\\
			&\quad + 2\LL y_1-y_2, \nabla f(x_1)y_1-\nabla f(x_2)y_2\RR + (p_0-1)\sum_{k = 1}^{m}|\nabla g_k(x_1)y_1-\nabla g_k(x_2)y_2|^2\notag\\
			& \le K_R|z_1-z_2|^2, \quad \forall\, |z_1|, ~ |z_2| \le R.
		\end{align}
		
		Similar to \eqref{tamed monotone}--\eqref{tamed g}, by \eqref{Dg}, it can be shown that for all $z\in \R^{2d}$ and $p< p_0$, 
		\begin{align}
			&\;2\LL z, \widetilde{F}(z)\RR + (p-1)\sum_{k = 1}^{m}|\widetilde{G}_k(z)|^2 \le K(1 + |z|^2), \\
			&\;|\widetilde{F}(z)| \le K(T)n^{\frac{1}{2}}(1 + |z|), \quad |\widetilde{G}_k(z)|^2 \le K(T)n^{\frac{1}{2}}(1 + |z|^2). \label{FG}
		\end{align}
		Denote $\mathcal{X}(t) := X_N^\alpha(t)-X_N^\alpha(\kappa_N(t))$, $t\in[0, T] $. We rewrite 
		\begin{align*}
			&	\psi_0^{\alpha, N}(t) := \begin{pmatrix}
				0\\-T^{\alpha}N^{(\alpha\wedge\frac{1}{2})-\alpha} \frac{f(X^\alpha_N(\kappa_N(t)))|X^\alpha_N(\kappa_N(t))|^{2l}}{1 + (\frac{T}{N})^{\alpha}|X^\alpha_N(\kappa_N(t))|^{2l}}
			\end{pmatrix}, \\
			&\psi_k^{\alpha, N}(t) := \begin{pmatrix}
				0\\-T^{\alpha}N^{(\alpha\wedge\frac{1}{2})-\alpha} \frac{g_k(X^\alpha_N(\kappa_N(t)))|X^\alpha_N(\kappa_N(t))|^{2l}}{1 + (\frac{T}{N})^{\alpha}|X^\alpha_N(\kappa_N(t))|^{2l}}
			\end{pmatrix}, \\
			&\psi_{m + 1, k}^{\alpha, N}(t) := \begin{pmatrix}
				0\\-N^{\alpha\wedge\frac{1}{2}}\sum_{u = 1}^{m}g_k(X^\alpha_N(\kappa_N(t)))g_u^{\alpha}(X^\alpha_N(\kappa_N(t)))\big(W_u(t)-W_u(\kappa_N(t))\big)
			\end{pmatrix}, \\
			&	r_1^{\alpha, N}(t) := \begin{pmatrix}
				0\\N^{\alpha\wedge\frac{1}{2}}\int_{0}^{1}(1-\lambda)\mathcal{D}^2f\big(X(t) + \lambda(X^\alpha_N(t)-X(t))\big)(X^\alpha_N(t)-X(t), X^\alpha_N(t)-X(t))
				 \ud \lambda
			\end{pmatrix}\\
			&\qquad\qquad\quad + N^{\alpha\wedge\frac{1}{2}}\begin{pmatrix}
				0\\ \int_{0}^{1} -(1-\lambda)\mathcal{D}^2 f(X^\alpha_N(\kappa_N(t)) + \lambda\mathcal{X}(t)) \big(\mathcal{X}(t), \mathcal{X}(t)\big ) \ud\lambda
			\end{pmatrix}, \\
			&r_2^{\alpha, N}(t) := \begin{pmatrix}
				0\\N^{\alpha\wedge\frac{1}{2}}\nabla f(X^\alpha_N(\kappa_N(t)))f^\alpha(X^\alpha_N(\kappa_N(t)))(\kappa_N(t)-t)
			\end{pmatrix}, \\
			&	r_{4, k}^{\alpha, N}(t) := \begin{pmatrix}
				0\\-N^{\alpha\wedge\frac{1}{2}} \nabla f(X^\alpha_N(\kappa_N(t)))g_k^\alpha(X^\alpha_N(\kappa_N(t)))(W_k(t)-W_k(\kappa_N(t)))
			\end{pmatrix}, \\
			&	r_{5, k}^{\alpha, N}(t) := \begin{pmatrix}
				0\\N^{\alpha\wedge\frac{1}{2}}\int_{0}^{1}(1-\lambda)\mathcal{D}^2g_k(X(t) + \lambda(X^\alpha_N(t)-X(t)))
				(X^\alpha_N(t)-X(t), X^\alpha_N(t)-X(t)) \ud \lambda
			\end{pmatrix}, \\
			&	r_{6, k}^{\alpha, N}(t) := \begin{pmatrix}
				0\\N^{\alpha\wedge\frac{1}{2}}\nabla g_k(X^\alpha_N(\kappa_N(t)))f^\alpha(X^\alpha_N(\kappa_N(t)))(\kappa_N(t)-t)
			\end{pmatrix}\\
			&\qquad\qquad\quad + \begin{pmatrix}
				0\\-N^{\alpha\wedge\frac{1}{2}}\int_{0}^{1}(1-\lambda)\mathcal{D}^2 g_k(X^\alpha_N(\kappa_N(t)) + \lambda\mathcal{X}(t))
				\big (\mathcal{X}(t), \mathcal{X}(t)\big ) \ud \lambda
			\end{pmatrix}, 
		\end{align*}
		where $k = 1, \ldots, m$, $t\in[0, T] $. By \eqref{stochatic fubini}, we have that $\sum_{k = 1}^{m}\int_{0}^{t}r_{4, k}^{\alpha, N}(s) \ud s = \begin{pmatrix}
			0\\R_3^{\alpha, N}(t)
		\end{pmatrix} + \begin{pmatrix}
			0\\R_4^{\alpha, N}(t)
		\end{pmatrix} $. 
		One then obtains that there exists $p_2\in(2, p_0)$ such that for any $p \le p_2$, 
		\begin{align} \label{Psibound}
			&\sup_{t\in[0, T]}\sum_{k = 0}^{m}\mbf E|\psi_k^{\alpha, N}(t)|^p< \infty, \quad \sup_{t\in[0, T]}\sum_{k = 1}^{m}\mbf E|\psi^{\alpha, N}_{m + 1, k}(t)|^p< \infty, \quad \sup_{t\in[0, T]}\mbf E|r_1^{\alpha, N}(t)|^p< \infty, \notag \\
			&	\sup_{t\in[0, T]}\mbf E|r_2^{\alpha, N}(t)|^p< \infty, \quad \sup_{t\in[0, T]}\sum_{i = 4}^{6}\sum_{k = 1}^{m}\mbf E|r_{i, k}^{\alpha, N}(t)|^p< \infty.
		\end{align}
		By Lemma \ref{UNexpression} and \eqref{Z2}, we have that
		\begin{align*}
			 \ud Z^{\alpha, N, n}(t) & = \Big(\widetilde{F}(Z^{\alpha, N, n}(\kappa_n(t))) + \psi_0^{\alpha, N}(t) + r_1^{\alpha, N}(t) + r_2^{\alpha, N}(t) + \sum_{k = 1}^{m}r_{4, k}^{\alpha, N}(t)\Big) \ud t\\
			&\quad + \sum_{k = 1}^{m}\Big(\widetilde{G}_k(Z^{\alpha, N, n}(\kappa_n(t))) + \psi_k^{\alpha, N}(t) + \psi^{\alpha, N}_{m + 1, k}(t) + r_{5, k}^{\alpha, N}(t) + r_{6, k}^{\alpha, N}(t)\Big) \ud W_k(t).
		\end{align*}
		For $2 \le p \le p_2$, It\^o's formula and Young's inequality give 
		\begin{align} \label{ZNnbound}
			&\ \mbf E|	Z^{\alpha, N, n}(t)|^p \notag\\
			& \le\mbf E|X_0|^p + \frac{p}{2}\mbf E\int_{0}^{t}\Big [|Z^{\alpha, N, n}(s)|^{p-2}\Big(\LL 2Z^{\alpha, N, n}(s), \widetilde{F}(Z^{\alpha, N, n}(\kappa_n(s)))\RR \notag\\
			&\quad + \LL 2Z^{\alpha, N, n}(s), \psi_0^{\alpha, N}(s) + r_1^{\alpha, N}(s) + r_2^{\alpha, N}(s) + \sum_{k = 1}^{m}r_{4, k}^{\alpha, N}(s)\RR\notag \\
			&\quad + (p-1)\sum_{k = 1}^{m}|\widetilde{G}_k(Z^{\alpha, N, n}(\kappa_n(s))) + \psi_k^{\alpha, N}(s) + \psi^{\alpha, N}_{m + 1, k}(s) + r_{5, k}^{\alpha, N}(s) + r_{6, k}^{\alpha, N}(s)|^2\Big)\Big ] \ud s\notag\\
			& \le\mbf E|X_0|^p + \frac{p}{2}\mbf E\int_{0}^{t}\Big[|Z^{\alpha, N, n}(s)|^{p-2}\LL 2Z^{\alpha, N, n}(s), \widetilde{F}(Z^{\alpha, N, n}(\kappa_n(s)))\RR \notag\\
			&\quad + (p-1)(1 + \epsilon_1)|Z^{\alpha, N, n}(s)|^{p-2}\sum_{k = 1}^{m}|\widetilde{G}_k(Z^{\alpha, N, n}(\kappa_n(s)))|^2\Big] \ud s\notag\\
			&\quad + K\mbf E\int_{0}^{t}|	Z^{\alpha, N, n}(s)|^p \ud s + K\int_{0}^{t}\mbf E\big(|r_1^{\alpha, N}(s)|^p + |r_2^{\alpha, N}(s)|^p + |\psi_0^{\alpha, N}(s)|^p\big) \ud s\notag\\
			&\quad + K\sum_{k = 1}^{m}\int_{0}^{t}\mbf E\big(|\psi_k^{\alpha, N}(s)|^p + |\psi_{m + 1, k}^{\alpha, N}(s)|^p + |r_{4, k}^{\alpha, N}(s)|^p + |r_{5, k}^{\alpha, N}(s)|^p + |r_{6, k}^{\alpha, N}(s)|^p\big) \ud s, 
		\end{align}
		where $\epsilon_1>0$ with $(p-1)(1 + \epsilon_1)< (p_0-1) $. Further, \eqref{FGbounded}--\eqref{FG} indicate that Conditions (A-3), (A-4), (B-2), and (B-3) of \cite{Sabanis2016} are fulfilled. Similar to the proof of \cite[Lemma 2]{Sabanis2016}, we infer that for any $p \le p_2$, $\sup_{n\in \mathbb{N}}\sup_{t\in[0, T]}\mbf E|Z^{\alpha, N, n}(t)|^p \le K(p, T, \mbf E|X_0|^{p_0}) $. This immediately implies that for any $p \le p_2$, 
		\begin{align}
			& \sup_{n\in \mathbb{N}}\sup_{t\in[0, T]}\mbf E|U^{\alpha, N, n}(t)|^p \le K(p, T, \mbf E|X_0|^{p_0}), \label{FGbound}\\
			& \sup_{n\in \mathbb{N}}\sup_{t\in[0, T]}\mbf E|X^n(t)|^p \le K(p, T, \mbf E|X_0|^{p_0}). \label{fgbound}
		\end{align}
		Similar to \eqref{ZNnbound}, it can be shown that for $2 \le p \le p_2$ and $\epsilon_2>0$ with $(p-1)(1 + \epsilon_2)<(p_0-1)$, 
		\begin{align*}
			&\ \mbf E|	Z^{\alpha, N}(t)|^p \\
			& \le\mbf E|X_0|^p + K\mbf E\int_{0}^{t}| Z^{\alpha, N}(s)|^p \ud s + \frac{p}{2}\mbf E\int_{0}^{t}\Big [|Z^{\alpha, N}(s)|^{p-2}\Big(\LL 2Z^{\alpha, N}(s), F(Z^{\alpha, N}(s)\RR \notag\\
			&\quad + (p-1)(1 + \epsilon_2)\sum_{k = 1}^{m}|G_k(Z^{\alpha, N}(s)|^2\Big)\Big ] \ud s + K\int_{0}^{t}\mbf E\big(|r_1^{\alpha, N}(s)|^p + |r_2^{\alpha, N}(s)|^p + |\psi_0^{\alpha, N}(s)|^p\big) \ud s\notag\\
			&\quad + K\sum_{k = 1}^{m}\int_{0}^{t}\mbf E\big(|\psi_k^{\alpha, N}(s)|^p + |\psi_{m + 1, k}^{\alpha, N}(s)|^p + |r_{4, k}^{\alpha, N}(s)|^p + |r_{5, k}^{\alpha, N}(s)|^p + |r_{6, k}^{\alpha, N}(s)|^p\big) \ud s.
		\end{align*}
		Then the application of \eqref{FGbounded}, \eqref{Psibound}, Gr\"onwall's inequality, and H\"older's inequality yields 
		\begin{align}
			\sup_{t\in[0, T]}\mbf E|Z^{\alpha, N}(t)|^p \le K(p, T, \mbf E|X_0|^{p_0}),\quad \text{for }p \le p_2.
		\end{align}
		Similar to the proof of Lemma \ref{TEMnumberical}, one then obtains that for any $t$, $s\in[0, T]$, 
		\begin{align}
			\mbf E|X^{n}(t)-X^{n}(s)|^4 \le K|t-s|^{2}.
		\end{align}
		Applying H\"older's inequality and the BDG inequality, combining Lemma \ref{TEMbound}, \eqref{FGbound}, and \eqref{fgbound}, we can derive that for any $t$, $s\in[0, T]$, 
		\begin{align}
			\mbf E|U^{\alpha, N, n}(t)-U^{\alpha, N, n}(s)|^4 \le K|t-s|^{2}.
		\end{align}
		Similar to the proof of Theorem \ref{TEMconverge}, we have 
		\begin{align} \label{Xnconvergence}
			\sup_{t \in[0, T]}\mbf E|X(t)-X^n(t)|^4 \le Kn^{-2}.
		\end{align}
		Then consider the function $\eta (t) := \exp(-(L + 2)t)$, $t\in[0, T]$, where $L$ is the constant in Assumption \ref{assum1}. Then It\^o's formula and Young's inequality yield that for $\epsilon_3>0$ with $(1 + \epsilon_3)< p_0-1$, 
		\begin{align} \label{UUU}
			&\ \mbf E(\eta (t)|U^{\alpha, N, n}(t)-U^{\alpha, N}(t)|^2) \notag\\
			& = \mbf E\int_{0}^{t}\Big [-(L + 2)\eta (s)|U^{\alpha, N, n}(s)-U^{\alpha, N}(s)|^2\notag \\
			&\quad + \eta(s)\Big(2\LL U^{\alpha, N, n}(s)-U^{\alpha, N}(s), \frac{\nabla f(X^n(\kappa_n(s)))U^{\alpha, N, n}(\kappa_n(s))}{1 + \frac{T}{n}(|X^n(\kappa_n(s))|^2 + |U^{\alpha, N, n}(\kappa_n(s))|^2)^l}-\nabla f(X(s))U^{\alpha, N}(s)\RR\notag \\
			&\quad + \sum_{k = 1}^{m}\Big |\frac{\nabla g_k(X^n(\kappa_n(s)))U^{\alpha, N, n}(\kappa_n(s))}{1 + \frac{T}{n}(|X^n(\kappa_n(s))|^2 + |U^{\alpha, N, n}(\kappa_n(s))|^2)^l}-\nabla g_k(X(s))U^{\alpha, N}(s)\Big |^2\Big)\Big ] \ud s\notag\\
			& \le \mbf E\int_{0}^{t}\Big [-(L + 2)\eta (s)|U^{\alpha, N, n}(s)-U^{\alpha, N}(s)|^2\notag\\
			&\quad + \eta(s)\Big(2\LL U^{\alpha, N, n}(s)-U^{\alpha, N}(s), \nabla f(X(s))(U^{\alpha, N, n}(s)-U^{\alpha, N}(s))\RR\notag\\
			&\quad + (1 + \epsilon_3)\sum_{k = 1}^{m}|\nabla g_k(X(s))(U^{\alpha, N, n}(s)-U^{\alpha, N}(s))|^2 + 2|U^{\alpha, N, n}(s)-U^{\alpha, N}(s)|^2\notag\\
			&\quad + |\nabla f(X^n(\kappa_n(s)))U^{\alpha, N, n}(\kappa_n(s))-\nabla f(X(s))U^{\alpha, N, n}(s)|^2\notag\\
			&\quad + K\sum_{k = 1}^{m}\big|\nabla g_k(X^n(\kappa_n(s)))U^{\alpha, N, n}(\kappa_n(s))-\nabla g_k(X(s))U^{\alpha, N, n}(s)\big|^2\notag\\
			&\quad + Kn^{-2}\sum_{k = 1}^{m}\Big|\nabla g_k(X^n(\kappa_n(s)))U^{\alpha, N, n}(\kappa_n(s))\big(|X^n(\kappa_n(s))|^2 + |U^{\alpha, N, n}(\kappa_n(s))|^2\big)^l\Big|^2\notag\\
			&\quad + Kn^{-2}\Big|\nabla f(X^n(\kappa_n(s)))U^{\alpha, N, n}(\kappa_n(s))\big(|X^n(\kappa_n(s))|^2 + |U^{\alpha, N, n}(\kappa_n(s))|^2\big)^l\Big|^2\Big)\Big] \ud s.
		\end{align}
		
		Combining \eqref{genbounded}, and $\mathcal{D}^2 f$, $\mathcal{D}^2 g_k\in\mathbf{F}$, $k = 1, \ldots, m$, we derive that 
		\begin{align} \label{Uniform}
			\mbf E\big (\eta (t)|U^{\alpha, N, n}(t)-U^{\alpha, N}(t)|^2\big ) \le \mbf E\int_{0}^{t}\xi^{\alpha, n} (s) \ud s, 
		\end{align}
		where
		\begin{align*}
			\xi^{\alpha, n} (t)
			& = K\eta(t)\Big (1 + |X(t)|^\iota + |X(\kappa_n(t))|^\iota + |X^n(\kappa_n(t))|^\iota + |U^{\alpha, N, n}(\kappa_n(t))|^\iota\\
			&\quad + |U^{\alpha, N, n}(t)|^\iota + |U^{\alpha, N}(\kappa_n(t))|^\iota + |U^{\alpha, N}(t)|^\iota \Big )\Big(|U^{\alpha, N, n}(t)-U^{\alpha, N, n}(\kappa_n(t))|^2 \\
			&\quad + |X(t)-X^n(t)|^2 + |X^n(\kappa_n(t))-X^n(t)|^2 + n^{-2}\Big), \quad t\in[0, T], 
		\end{align*}
		for some $\iota>0 $ dependent on the growing degree of $\mathcal{D}^2 f$. By \eqref{FGbound}--\eqref{Xnconvergence} and H\"older's inequality, it can be shown that $\sup_{t\in[0, T]}\mbf E\xi^{\alpha, n} (t) \le Kn^{-1}$. Similar to \eqref{UUU} and \eqref{Uniform}, It\^o's formula also gives that
		for every stopping time $\tau \le T$, $\mbf E(\eta (\tau)|U^{\alpha, N, n}(\tau)-U^{\alpha, N}(\tau)|^2) \le \mbf E\int_{0}^{\tau}\xi^{\alpha, n} (s) \ud s $. 
		Applying Lemma \ref{uniform lp}, it holds that
		\begin{align*}
			\mbf E\sup_{t \in[0, T]}\big (\eta (t)|U^{\alpha, N, n}(t)-U^{\alpha, N}(t)|^2\big )^\gamma \le \mbf E\Big (\int_{0}^{T}\xi^{\alpha, n} (s) \ud s\Big )^\gamma, 
		\end{align*}
		for any $\gamma\in (0, 1) $. Then we infer 
		\begin{align*}
			\mbf E\sup_{t \in[0, T]}\big (|U^{\alpha, N, n}(t)-U^{\alpha, N}(t)|^2\big )^\gamma& \le \exp(\gamma(L + 2)T)\mbf E\sup_{t \in [0, T]}\big (\eta (t)|U^{\alpha, N, n}(t)-U^{\alpha, N}(t)|^2\big )^\gamma\\
			& \le K\Big (\mbf E \int_{0}^{T}\xi^{\alpha, n} (s) \ud s\Big )^\gamma \le Kn^{-\gamma}, 
		\end{align*}
		with $K$ being independent of $N$. By taking $\gamma = \frac{1}{2}$, one has $\lim_{n\to \infty}\sup_{N \geq 1}\mbf E \|U^{\alpha, N, n}-U^{\alpha, N}\|_{\mathbf{C}([0, T])} = 0$, which completes the proof. 
	\end{proof}

	With previous preparation, now we can give the limit distribution of $U^{\alpha, N}(t)$. 
	
	\begin{theo} \label{maintheorem1}
		Let Assumptions \ref{assum1} and \ref{assum2} hold. Then for any $t\in[0, T]$, we have that $U^{\alpha, N}(t) \overset{d}{\Longrightarrow} U^\alpha(t)$ as $N\to\infty$, and thus $N^{\frac{1}{2}\wedge\alpha}(\bar{X}^\alpha_N-X(T)) \overset{d}{\Longrightarrow} U^\alpha(T)$. Here, $U^\alpha = \{U^\alpha(t)\}_{t\in[0, T]}$ is the strong solution of the following equation
		\begin{align} \label{Ulimit}
			U^\alpha(t)
			& = \int_0^t\nabla f(X(s))U^{\alpha}(s) \ud s + \sum_{k = 1}^{m}\int_0^t\nabla g_k(X(s))U^{\alpha}(s) \ud W_k(s)\notag\\
			&\quad-\mathbbm 1_{\{0<\alpha \le \frac{1}{2}\}}T^{\alpha}\int_0^tf(X(s))|X(s)|^{2l} \ud s\notag-\mathbbm 1_{\{0<\alpha \le \frac{1}{2}\}}T^{\alpha}\sum_{k = 1}^{m}\int_0^tg_k(X(s))|X(s)|^{2l} \ud W_k(s)\\
			&\quad + \mathbbm 1_{\{\frac{1}{2} \le \alpha \le 1\}}\frac{\sqrt{2T}}{2}\sum_{k = 1}^{m}\sum_{u = 1}^{m}\int_{0}^{t}\nabla g_k(X(s))g_u(X(s)) \ud \widetilde{W}_{ku}(s), 
		\end{align}
		where $\mbf {\widetilde{W}} = (\widetilde{W}_{11}, \widetilde{W}_{12}, \ldots, \widetilde{W}_{1m}, \ldots, \widetilde{W}_{mm})$ is an $m^2$-dimensional standard Brownian motion independent of $\mbf W$.
	\end{theo}

	\begin{proof}
		The proof is mainly based on \cite[Theorem 3.2]{HJWY24}. Lemma \ref{A1} indicates that $U^{\alpha, N, n}$ and $U^{\alpha, N}$ satisfy Condition (A1) of \cite[Theorem 3.2]{HJWY24}.
		
		For any fixed $n\in\mbb N^ + $, define the mapping $\Gamma^n:\mbf C([0, T];\mbb R^{d})\times\mbf C([0, T];\mbb R)^{\otimes m}\times\mbf C([0, T];\mbb R^{d})\to \mbf C([0, T];\mbb R^{d})$ which maps $(u, q_1, \ldots, q_m, h)$ to the solution of
		\begin{align*}
			p(t) & = \int_0^t\frac{\nabla f(u(\kappa_n(s)))p(\kappa_n(s))}{1 + \frac{T}{n}(|u(\kappa_n(s))|^2 + |p(\kappa_n(s))|^2)^l} \ud s\\
			&\quad + \sum_{i = 1}^m\int_0^t\frac{\nabla g_i(u(\kappa_n(s)))p(\kappa_n(s))}{1 + \frac{T}{n}(|u(\kappa_n(s))|^2 + |p(\kappa_n(s))|^2)^l} \ud q_i(s) + h(t), \quad t\in[0, T].
		\end{align*}
		Following the argument for the continuity of $F^\Delta$ in the proof of \cite[Theorem 4.3]{MAM}, we have that for any fixed $n\in\mbb N^ + $, $\Gamma^n$ is continuous w.r.t.\ $(u, q_1, \ldots, q_m, h)$. Further, it holds that $(X^n, W_1, \ldots, W_m, \theta ^{\alpha, N})\overset{stably}{\Longrightarrow}(X^n, W_1, \ldots, W_m, \theta ^{\alpha})$ as $N\to \infty$, due to Proposition \ref{Pro1} and Lemma \ref{IN41}. Thus, $(X^n, W_1, \ldots, W_m, \theta^{\alpha, N}) \overset{d}{\Longrightarrow} (X^n, W_1, \ldots, W_m, \theta ^{\alpha})$ as $N\to \infty $. The continuous mapping theorem gives that $\Gamma^n(X^n, W_1, \ldots, W_m, \theta^{\alpha, N}) \overset{d}{\Longrightarrow} \Gamma^n(X^n, W_1, \ldots, W_m, \theta^\alpha)$ in $\mbf C([0, T];\mbb R^{d})$ as $N\to\infty$. By the definition of $\Gamma^n$, we have that $\Gamma^n(X^n, W_1, \ldots, W_m, \theta^{\alpha, N}) = U^{\alpha, N, n}$ and that $U^{\alpha, \infty, n} := \Gamma^n(X^n, W_1, \ldots, W_m, \theta^\alpha)$ is the strong solution of
		\begin{align*}
			U^{\alpha, \infty, n}(t)& = \int_0^t\frac{\nabla f(X^n(\kappa_n(s)))U^{\alpha, \infty, n}(\kappa_n(s))}{1 + \frac{T}{n}(|X^n(\kappa_n(s))|^2 + |U^{\alpha, \infty, n}(\kappa_n(s))|^2)^l} \ud s\\
			&\quad + \sum_{i = 1}^m\int_0^t \frac{\nabla g_i(X^n(\kappa_n(s)))U^{\alpha, \infty, n}(\kappa_n(s))}{1 + \frac{T}{n}(|X^n(\kappa_n(s))|^2 + |U^{\alpha, \infty, n}(\kappa_n(s))|^2)^l} \ud W_i(s) + \theta^\alpha(t), \quad t\in[0, T].	
		\end{align*}
		Similar to the proof of Lemma \ref{A1}, one can show $\widetilde{\mbf E} \| U^{\alpha, \infty, n} - U^\alpha \|_{\mbf C([0, T])} \le Kn^{-\frac{1}{2}}$. In this way, we have 
		\begin{align*}
			U^{\alpha, N, n} \overset{d}{\Longrightarrow} U^{\alpha, \infty, n}~\text{as}~N\to\infty~\text{for given}~n, \quad U^{\alpha, \infty, n} \overset{d}{\Longrightarrow} U^\alpha~\text{as}~n\to\infty.
		\end{align*}
		Consequently, Conditions (A2) and (A3) of \cite[Theorem 3.2]{HJWY24} are fulfilled. Finally, the proof is complete through the application of \cite[Theorem 3.2]{HJWY24}.
	\end{proof}

	\begin{cor} \label{Cor1}
		Let Assumptions \ref{assum1} and \ref{assum2} hold. Then there exist $C_1$, $C_2$, $C_3>0$ independent of T such that for $\alpha \in[\frac{1}{2}, 1]$, and $t\in[0, T]$, 
		\begin{align}
			\mbf E|U^\alpha(t)|^2 \le (C_1 + C_2\mathbbm1_{\{\alpha = \frac{1}{2}\}})e^{C_3T}T^2.
		\end{align}
	\end{cor}

	\begin{proof}
		In this proof, we denote by $C$ a generic constant independent of $T$, which may vary for each appearance. By \eqref{Ulimit}, It\^o's formula gives that for $\alpha\in[\frac{1}{2}, 1]$, 
		\begin{align*}
			\mbf E|U^\alpha(t)|^2
			& = \mbf E\int_{0}^{t}\Big[2\LL U^\alpha(s), \nabla f(X(s))U^\alpha(s)-\mathbbm 1_{\{\alpha = \frac{1}{2}\}}T^{\frac{1}{2}}f(X(s))|X(s)|^{2l}\RR\\
			&\quad + \sum_{k = 1}^{m}\Big |\nabla g_k(X(s))U^\alpha(s)-\mathbbm 1_{\{\alpha = \frac{1}{2}\}}T^{\frac{1}{2}}g_k(X(s))|X(s)|^{2l}\Big |^2 \\
			&\quad + \frac{T}{2}\sum_{k = 1}^{m}\sum_{u = 1}^{m}|\nabla g_k(X(s))g_u(X(s))|^2\Big] \ud s.
		\end{align*}
		Then Young's inequality, \eqref{f l}, \eqref{g l}, \eqref{genbounded}, and \eqref{Dg} yield 
		\begin{align*}
			&\ \mbf E|U^\alpha(t)|^2 \\ 
			& \le C\mbf E\int_{0}^{t}|U^\alpha(s)|^2 \ud s + C\mathbbm1_{\{\alpha = \frac{1}{2}\}}T\mbf E\int_{0}^{t} |f(X(s))|^2|X(s)|^{4l} \ud s\\
			&\quad + C\mathbbm1_{\{\alpha = \frac{1}{2}\}}T\mbf E\int_{0}^{t}\sum_{k = 1}^{m}|g_k(X(s))|^2|X(s)|^{4l} \ud s + CT\sum_{k = 1}^{m}\sum_{u = 1}^{m}\mbf E\int_{0}^{t}|\nabla g_k(X(s))g_u(X(s))|^2 \ud s\\
			& \le C\int_{0}^{t}\mbf E|U^\alpha(s)|^2 \ud s + C\mathbbm1_{\{\alpha = \frac{1}{2}\}}T\int_{0}^{t}\mbf E|X(s)|^{\iota} \ud s + CT\int_{0}^{t}\mbf E|X(s)|^{\iota} \ud s + CT^2, 
		\end{align*}
		for some $\iota>0 $ dependent on the growing degree of $f$ and $g_k$, $k = 1, \ldots, m$. It follows from \eqref{Ybound} that 
		\begin{align*}
			&\mbf E|U^\alpha(t)|^2 \le C\int_{0}^{t}\mbf E|U^\alpha(s)|^2 \ud s + (C\mathbbm1_{\{\alpha = \frac{1}{2}\}}e^{CT} + Ce^{CT})T^2.
		\end{align*}
		Thus, the proof is complete by applying Gr\"onwall's inequality.
	\end{proof}

	\begin{rem} \label{RA}
		Since $N^{\alpha\wedge\frac{1}{2}}(\bar{X}_N^\alpha-X(T)) \overset{d}{\Longrightarrow} U^\alpha(T)$, $\mbf E|\bar{X}_N^\alpha-X(T)|^2\approx\frac{1}{N^{2\alpha\wedge1}}\mbf E|U^\alpha(T)|^2$ for $N\gg1$. Thus, to some extent, Theorem \ref{maintheorem1} and Corollary \ref{Cor1} indicate that $\alpha$ is the key parameter
		reflecting the growth rate of mean-square error of the tamed Euler method \eqref{TEM}. In addition, we infer that for the tamed Euler method \eqref{TEM} of strong order $\frac{1}{2}$, the one with $\alpha = \frac{1}{2}$ has the largest mean-square error after a long time.
	\end{rem}

	\section{Asymptotic error distribution for additive noise} \label{Secadd}
	
	In this section, we investigate the asymptotic error distribution of a class of tamed Euler methods applied to SDEs with additive noise. 
	
	Consider the following SDE:\ 
	\begin{align} \label{mulSDE2}
		\begin{cases}
			 \ud Y(t) = f(Y(t)) \ud t + \sigma \ud \mbf W(t), \qquad t\in(0, T], \\
			Y(0) = Y_0, 
		\end{cases}
	\end{align}
	where $f:\R^d\to \R^d$, $\sigma = (\sigma_1, \ldots, \sigma_m)\in \R^{d\times m}$ is a constant matrix, and $\mbf W = (W_1, W_2, \ldots, W_m)$ is an $m$-dimensional standard Brownian motion. Then we introduce the following assumption.
	
	\begin{assum} \label{assum3}
		Assume that the following conditions hold.
		\begin{itemize}
			
			\item [(B-1)] $f\in \mathbf{C}^2(\R^d)$, $\mathcal{D}^2f\in \mathbf{F}$, and there exists a constant $L_2>0$ such that 
			$$ 2\LL x-y, f(x) -f(y)\RR \le L_2|x-y|^2, \quad \forall \, x, \, y\in \R^d. $$
			
			\item [(B-2)] $Y_0 \in \mbf L^{p}(\Omega, \mathcal F, \mbf P;\mbb R^d)$ for all $p \geq 1$.
		\end{itemize}
	\end{assum}
	
	\begin{rem}
		It follows from Assumption \ref{assum3} that there exist $L_3>0$ and $l>0$ such that
		\begin{align} 
			& 2\LL x, f(x) \RR \le L_3(1 + {|x|}^2), \\
			& |f(x) -f(y)| \le L_3(1 + |x|^l + |y|^l)|x-y|, 
		\end{align} 
		for all $x$, $y\in\R^d$. Further, under Assumption \ref{assum3}, \eqref{mulSDE2} admits a unique strong solution given by
		\begin{align} \label{functionY}
			Y(t) = Y_0 + \int_{0}^{t}f(Y(s)) \ud s + \int_{0}^{t}\sigma \ud \mathbf{W}(s), \quad t\in[0, T].
		\end{align}
		Moreover, for any $p>0$, and $t\in[0, T]$, 
		\begin{align}
			\|Y(t)\|_{\mbf L^p(\Omega)}& \le C_2(p)e^{C_2(p)T} \big( 1 + \|Y_0\|_{\mbf L^p(\Omega)} \big),
		\end{align}
		where $C_2(p)$ depends on $p$ but is independent of $T$ (cf.\ \cite[Chapter 2.4]{Maoxuerong}).
	\end{rem}

	\subsection{Strong convergence of tamed Euler method for additive noise}
	
	In this section, we establish the strong convergence rate of the tamed Euler method for additive noise. 
	
	Consider the following tamed Euler method for \eqref{mulSDE2}:
	\begin{align} \label{TEMA}
		\bar{Y}_{n + 1}^{\alpha} = \bar{Y}_n^{\alpha} + \frac{T}{N}f^{\alpha}(\bar{Y}_n^{\alpha}) + \sigma \left(\mbf W\big(\frac{(n + 1)T}{N}\big)-\mbf W\big(\frac{nT}{N}\big)\right), 
	\end{align}
	for $n = 0, 1, \ldots, N-1$, $N\in \mathbb N$, where $\bar{Y}_0^{\alpha} = Y_0$, $f^{\alpha}(x) := \frac{f(x) }{1 + (\frac{T}{N})^{\alpha}|x|^{l_\alpha}}$, $x\in\R^d$, with $l_\alpha = \lceil 2\alpha\rceil l$, and $\alpha>0 $. 
	
	Introduce the continuous version of ${\lbrace \bar {Y}_n^{\alpha}\rbrace}_{n = 0}^N$:\ 
	\begin{align} \label{CTEMA}
		Y^\alpha_N(t) = Y_0 + \int_{0}^{t} f^{\alpha}(Y^\alpha_N(\kappa_N(s))) \ud s + \int_{0}^{t} \sigma \ud \mbf W(s), 
	\end{align}
	where $\kappa_N(t) = \lfloor\frac{Nt}{T}\rfloor\frac{T}{N}$, $t\in[0, T]$.

	\begin{theo} \label{TEMAconverge}
		Let Assumption \ref{assum3} hold. Then for any $p>0$, the solution of \eqref{CTEMA} converges to that of the SDE \eqref{mulSDE2} in $\mathbf{L}^p$-sense with order $\alpha\wedge1$, that is, 
		\begin{align}
			\sup_{t\in[0, T]}\mbf E|Y^\alpha_N(t)-Y(t)|^p \le KN^{-(\alpha \wedge1)p }.
		\end{align}
	\end{theo}

	\begin{proof}
		Similar to the proof of Lemma \ref{TEMbound}, it can be shown that for any $p>0$, 
		\begin{align}
			\sup_{N\in \mathbb{N}}\sup_{t\in[0, T]}\mathbf{E}|Y^\alpha_N(t)|^p \le K. \label{AA1}
		\end{align}
		Consider the one step approximation of numerical scheme \eqref{TEMA}:
		\begin{align*}
			Y_{t, x}^\alpha(t + h) = x + \frac{hf(x) }{1 + h^{\alpha}|x|^{l_\alpha}} + \sigma\big(\mbf W(t + h)-\mbf W(t)\big), 
		\end{align*}
		where $h>0$ and $0 \le t \le T-h $. Let $Y_{t, x}(t + h)$ be the strong solution of 
		\begin{align} \label{Ytx}
			Y_{t, x}(t + h) = x + \int_{t}^{t + h}f(Y_{t, x}(s)) \ud s + \int_{t}^{t + h}\sigma \ud \mbf W(s), 
		\end{align}
		where $h>0$ and $0 \le t \le T-h $. 
		One observes that $	Y_{t, x}(t + h)-Y_{t, x}^\alpha(t + h) = \int_{t}^{t + h}\big [f(Y_{t, x}(s))-f(x) \big ] \ud s + h\big (f(x) -\frac{f(x) }{1 + h^{\alpha}|x|^{l_\alpha}}\big )$.
		The application of Taylor's formula and \eqref{Ytx} yields 
		\begin{align*}
			&\ f(Y_{t, x}(s))-f(x) \\
			& = \nabla f(x) \left( \int_{t}^{s}f(Y_{t, x}(r)) \ud r + \int_{t}^{s}\sigma \ud \mbf W(r) \right) \\
			&\quad + \int_{0}^{1}(1-\lambda)\mathcal{D}^2f\big (x + \lambda(Y_{t, x}(t + h)-x)\big )(Y_{t, x}(t + h)-x, Y_{t, x}(t + h)-x) \ud \lambda.
		\end{align*}
		Then it can be shown that 
		\begin{align}
			& \big |\mbf E\big(Y_{t, x}(t + h)-Y_{t, x}^\alpha(t + h)\big)\big| \le K(1 + |x|^\iota)(h^{\alpha + 1} + h^2), \label{AA2}\\
			& \big (\mbf E\big|Y_{t, x}(t + h)-Y_{t, x}^\alpha(t + h)\big|^p\big)^{\frac{1}{p}} \le K(1 + |x|^\iota)(h^{\alpha + 1} + h^{\frac{3}{2}} + h^2), \quad p \geq 2, \label{AA3}
		\end{align}
		for some $\iota>0 $ dependent on $\alpha$ and the growing degree of $ f$. 
		By \eqref{AA1}, \eqref{AA2}, and \eqref{AA3}, we deduce from \cite[Theorem 2.1]{ZhangZQ13} that for any $p \geq 2$, 
		\begin{align*}
			\sup_{t \in[0, T]}\mbf E\big|Y(\kappa_N(t))-Y_{N}^\alpha(\kappa_N(t))\big|^p \le KN^{-p(\alpha\wedge 1)}.
		\end{align*}
		Applying \eqref{functionY} and \eqref{CTEMA}, we have 
		\begin{align*}
			Y(t)-Y^\alpha_N(t) = Y(\kappa_N(t))-Y_{N}^\alpha(\kappa_N(t)) + \int_{\kappa_N(t)}^{t}\big[f(Y(s))- f^{\alpha}(Y^\alpha_N(\kappa_N(s)))\big] \ud s.
		\end{align*}
		Therefore for any $p \geq 2$, 
		\begin{align*}
			\sup_{t\in[0, T]}\mbf E|Y^\alpha_N(t)-Y(t)|^p \le KN^{-(\alpha \wedge1)p }.
		\end{align*}
		Finally, the proof is complete by H\"older's inequality.
	\end{proof}

	\subsection{Asymptotic error distribution for additive noise}
	
	In this section, we derive the asymptotic error distribution for additive noise, which requires the following assumption.

	\begin{assum} \label{assum4}
		Assume that $f\in\mbf C^3(\mbb R^d)$ and $\mathcal D^3f\in\mbf F$.
	\end{assum}

	We are in the position to give an expansion for the normalized error $N^{\alpha\wedge1}( \bar{Y}^\alpha_N(t)-Y(t))$.
	
	\begin{lem} 
		Denote $V^{\alpha, N}(t) := N^{\alpha\wedge1}( Y^\alpha_N(t)-Y(t))$, $t\in[0, T]$. Let Assumptions \ref{assum3} and \ref{assum4} hold. Then $V^{\alpha, N}$ has the following representation
		\begin{align} \label{YY0}
			V^{\alpha, N}(t)& = \int_0^t\nabla f(Y(s))V^{\alpha, N}(s) \ud s + \sum_{i = 0}^{3}J_i^{\alpha, N}(t) + \widetilde{R}^{\alpha, N}(t), \quad t\in[0, T]. 
		\end{align}	
		Here, 
		\begin{align*}
			& J_0^{\alpha, N}(t) = -T^{\alpha} N^{(\alpha\wedge1)-\alpha}\int_0^t\frac{f(Y^\alpha_N(\kappa_N(s)))|Y^\alpha_N(\kappa_N(s))|^{l_\alpha}}{1 + (\frac{T}{N})^{\alpha}|Y^\alpha_N(\kappa_N(s))|^{l_\alpha}} \ud s, \\
			& J_1^{\alpha, N}(t) = -N^{\alpha\wedge1}\sum_{k = 1}^{m}\int_{0}^{t}\nabla f(Y^\alpha_N(\kappa_N(s)))(\kappa_N(s) + \frac{T}{N}-s)\sigma_k \ud W_k(s), \\
			& J_{2}^{\alpha, N}(t) = -\frac{1}{2}N^{\alpha\wedge1}\sum_{k = 1}^{m}\int_{0}^{t}\mathcal{D}^2 f(Y^\alpha_N(\kappa_N(s)))(\sigma_k, \sigma_k)(s-\kappa_N(s)) \ud s, \\
			& J_{3}^{\alpha, N}(t) = -N^{\alpha\wedge1}\int_{0}^{t}\nabla f(Y^\alpha_N(\kappa_N(s)))f(Y^\alpha_N(\kappa_N(s)))(s-\kappa_N(s)) \ud s, 
		\end{align*}
		and $\widetilde{R}^{\alpha, N}(t)$ is the remainder with $\lim_{N\to\infty}\mbf E\|\widetilde{R}^{\alpha, N}\|^2_{\mathbf{C}([0, T])} = 0 $. 
	\end{lem}

	\begin{proof}
	Combining \eqref{functionY} and \eqref{CTEMA}, we have 
		\begin{align} \label{V0}
			V^{\alpha, N}(t) = N^{\alpha\wedge1}\int_0^t \big [f^\alpha(Y^\alpha_N(\kappa_N(s)))-f(Y(s))\big ] \ud s.
		\end{align}
		Taylor's formula yields
		\begin{align}
			&\ N^{\alpha\wedge1}\int_0^t \big [f^\alpha(Y^\alpha_N(\kappa_N(s)))-f(Y(s))\big ] \ud s \notag\\
			& = \int_0^t\nabla f(Y(s))V^{\alpha, N}(s) \ud s + N^{\alpha\wedge1}\int_{0}^{t}\nabla f(Y^\alpha_N(\kappa_N(s)))(Y^\alpha_N(\kappa_N(s))-Y^\alpha_N(s)) \ud s\notag\\
			&\quad-N^{\alpha\wedge1}\int_{0}^{t}\frac{1}{2}\mathcal{D}^2 f(Y^\alpha_N(\kappa_N(s)))\big (Y^\alpha_N(\kappa_N(s))-Y^\alpha_N(s), Y^\alpha_N(\kappa_N(s))-Y^\alpha_N(s)\big ) \ud s\notag\\
			&\quad + \widetilde{R}_1^{\alpha, N}(t) + \widetilde{R}_2^{\alpha, N}(t) + J_0^{\alpha, N}(t), 
		\end{align}
		where
		\begin{align*}
			&\widetilde{R}_1^{\alpha, N}(t) = N^{\alpha\wedge1}\int_{0}^{t}\int_{0}^{1}(1-\lambda)\mathcal{D}^2f\big(Y(s) + \lambda(Y^\alpha_N(s)-Y(s))\big)(Y^\alpha_N(s)-Y(s), Y^\alpha_N(s)-Y(s)) \ud \lambda \ud s, \\
			&\widetilde{R}_2^{\alpha, N}(t) = N^{\alpha\wedge1}\int_{0}^{t}\int_{0}^{1}\frac{(1-\lambda)^2}{2}\mathcal{D}^3 f(Y^\alpha_N(\kappa_N(s)) + \lambda(Y^\alpha_N(s)-Y^\alpha_N(\kappa_N(s))))\\
			&\;\;\;\;\;\;\;\;\;\;\;\;\;\;\;\;\;\;\;\;\;\;\big (Y^\alpha_N(\kappa_N(s))-Y^\alpha_N(s), Y^\alpha_N(\kappa_N(s))-Y^\alpha_N(s), Y^\alpha_N(\kappa_N(s))-Y^\alpha_N(s)\big ) \ud \lambda \ud s, ~ t\in[0, T], 
		\end{align*}
		with $\mbf E\|\widetilde{R}_1^{\alpha, N}\|_{\mathbf{C}([0, T])}^2 \le KN^{-2(\alpha\wedge1)}$, and $\mbf E\|\widetilde{R}_2^{\alpha, N}\|^2_{\mathbf{C}([0, T])} \le KN^{-1}$. Due to \eqref{CTEMA}, it holds that
		\begin{align}
			&\ N^{\alpha\wedge1}\int_{0}^{t}\nabla f(Y^\alpha_N(\kappa_N(s))) (Y^\alpha_N(\kappa_N(s))-Y^\alpha_N(s)) \ud s \notag\\
			& = J_3^{\alpha, N}(t) - N^{\alpha\wedge1}\sum_{k = 1}^{m}\int_{0}^{t}\nabla f(Y^\alpha_N(\kappa_N(s)))\sigma_k\int_{\kappa_N(s)}^{s} \ud W_k(r) \ud s.
		\end{align}
		Then it follows from stochastic Fubini theorem that 
		\begin{align}
			&\ -N^{\alpha\wedge1}\sum_{k = 1}^{m}\int_{0}^{t}\nabla f(Y^\alpha_N(\kappa_N(s)))\sigma_k\int_{\kappa_N(s)}^{s} \ud W_k(r) \ud s\notag\\
			& = -N^{\alpha\wedge1}\sum_{k = 1}^{m}\int_{0}^{t}\int_{r}^{(\kappa_N(r) + T/N)\wedge t}\nabla f(Y^\alpha_N(\kappa_N(s)))\sigma_k \ud s \ud W_k(r)\notag\\
			& = J_1^{\alpha, N}(t) + \widetilde{R}_3^{\alpha, N}(t), 
		\end{align}
		where $\widetilde{R}_3^{\alpha, N}(t) := -N^{\alpha\wedge1}\sum_{k = 1}^{m}\int_{0}^{t}\nabla f(Y^\alpha_N(\kappa_N(r)))\sigma_k\big ((\kappa_N(r) + T/N)\wedge t-(\kappa_N(r) + \frac{T}{N})\big ) \ud W_k(r)$.
		It is clear that $\widetilde{R}_3^{\alpha, N}(t) = -N^{\alpha\wedge1}\sum_{k = 1}^{m}\int_{\kappa_N(t)}^{t}\nabla f(Y^\alpha_N(\kappa_N(r)))\sigma_k\big ( t-(\kappa_N(t) + \frac{T}{N})\big ) \ud W_k(r)$, due to the definition of $\kappa_N $. Similar to the estimate for $R_4^{\alpha, N}$ (see \eqref{R2}), we have $\mbf E\|\widetilde{R}_3^{\alpha, N}\|^2_{\mathbf{C}([0, T])} \le KN^{(2\alpha\wedge2)-\frac{9}{4}}$. We derive from \eqref{CTEMA}, stochastic Fubini theorem, and It\^o's formula that
		\begin{align} \label{YY4}
			&\ -N^{\alpha\wedge1}\int_{0}^{t}\frac{1}{2}\mathcal{D}^2 f(Y^\alpha_N(\kappa_N(s)))\big (Y^\alpha_N(\kappa_N(s))-Y^\alpha_N(s), Y^\alpha_N(\kappa_N(s))-Y^\alpha_N(s)\big ) \ud s\notag\\
			& = -N^{\alpha\wedge1}\int_{0}^{t}\frac{1}{2}\mathcal{D}^2 f(Y^\alpha_N(\kappa_N(s)))(f^\alpha(Y^\alpha_N(\kappa_N(s))), f^\alpha(Y^\alpha_N(\kappa_N(s))))(s-\kappa_N(s))^2 \ud s\notag\\
			&\quad-N^{\alpha\wedge1}\sum_{k = 1}^{m}\sum_{u = 1}^{m}\int_{0}^{t}\frac{1}{2}\mathcal{D}^2 f(Y^\alpha_N(\kappa_N(s)))(\sigma_u, \sigma_k)(W_k(s)-W_k(\kappa_N(s)))(W_u(s)-W_u(\kappa_N(s))) \ud s\notag\\
			&\quad-N^{\alpha\wedge1}\sum_{k = 1}^{m}\int_{0}^{t}\mathcal{D}^2 f(Y^\alpha_N(\kappa_N(s)))(f^\alpha(Y^\alpha_N(\kappa_N(s))), \sigma_k)(s-\kappa_N(s))(W_k(s)-W_k(\kappa_N(s))) \ud s\notag\\
			& = J_2^{\alpha, N}(t) + \widetilde{R}_4^{\alpha, N}(t), 
		\end{align}
		where
		\begin{align*}
			&	\widetilde{R}_4^{\alpha, N}(t) = -\frac{N^{\alpha\wedge1}}{2}\int_{0}^{t}\mathcal{D}^2 f(Y^\alpha_N(\kappa_N(s)))(f^\alpha(Y^\alpha_N(\kappa_N(s))), f^\alpha(Y^\alpha_N(\kappa_N(s))))(s-\kappa_N(s))^2 \ud s\\
			&-N^{\alpha\wedge1}\sum_{k = 1}^{m}\int_{0}^{t}\mathcal{D}^2 f(Y^\alpha_N(\kappa_N(s)))(f^\alpha(Y^\alpha_N(\kappa_N(s))), \sigma_k)(s-\kappa_N(s))(W_k(s)-W_k(\kappa_N(s))) \ud s\\
			&-N^{\alpha\wedge1}\sum_{1 \le k<u \le m}\int_{0}^{t}\int_{r}^{\kappa_N(r) + \frac{T}{N}}\mathcal{D}^2 f(Y^\alpha_N(\kappa_N(r)))(\sigma_u, \sigma_k)\big(W_k(s)-W_k(\kappa_N(r))\big ) \ud s \ud W_u(r)\\
			& + N^{\alpha\wedge1}\sum_{1 \le k<u \le m}(W_u(t)-W_u(\kappa_N(t)))\mathcal{D}^2 f(Y^\alpha_N(\kappa_N(t)))(\sigma_u, \sigma_k)\int_{t}^{\kappa_N(t) + \frac{T}{N}}\big(W_k(s)-W_k(\kappa_N(t))\big) \ud s\\
			&-N^{\alpha\wedge1}\sum_{k = 1}^{m}\int_{0}^{t}\mathcal{D}^2f(Y^\alpha_N(\kappa_N(r)))(\sigma_k, \sigma_k)\big(W_k(r)-W_k(\kappa_N(r))\big )(\kappa_N(r) + \frac{T}{N}-r) \ud W_k(r)\\
			&-\frac{N^{\alpha\wedge1}}{2}\sum_{k = 1}^{m}\mathcal{D}^2f(Y^\alpha_N(\kappa_N(t)))(\sigma_k, \sigma_k)(t-\kappa_N(t)-\frac{T}{N})\Big((W_k(t)-W_k(\kappa_N(t)))^2-(t-\kappa_N(t))\Big).
		\end{align*}
		Similar to the the proof of \eqref{R2} in Lemma \ref{UNexpression}, we have that $\mbf E\|\widetilde{R}_4^{\alpha, N}\|^2_{\mathbf{C}([0, T])} \le KN^{(2\alpha\wedge2)-\frac{9}{4}} $. 
		
		Combining \eqref{V0}--\eqref{YY4}, we get \eqref{YY0} with $\widetilde{R}^{\alpha, N}(t) = \sum_{i = 1}^{4}\widetilde{R}_i^{\alpha, N}(t)$. According to the previous estimation for $\widetilde{R}_i^{\alpha, N}$, $i = 1, 2, 3, 4$, we have $\lim_{N\to\infty}\mbf E\|\widetilde{R}^{\alpha, N}\|^2_{\mathbf{C}([0, T])} = 0 $. 
		Thus, the proof is complete.
	\end{proof}

	The following lemma gives the convergence of $J_2^{\alpha, N}$ and $J_3^{\alpha, N}$ in probability in $\mbf C([0, T];\R^d)$, which can not be directly obtained via \cite[Proposition 4.2]{HJWY24}.

	\begin{lem} \label{J2C}
		Let Assumptions \ref{assum3} and \ref{assum4} hold. Then $J_2^{\alpha, N}\overset{\mbf P}{\longrightarrow}J_2^{\alpha}$ and $J_3^{\alpha, N}\overset{\mbf P}{\longrightarrow}J_3^{\alpha}$ as $\mbf C([0, T];\mbb R^d)$-valued random variables as $N\to\infty$. Here, $J_{2}^{\alpha}$ and $J_{3}^{\alpha}$ are the stochastic processes respectively defined by
		\begin{align*}
			&J_{2}^{\alpha}(t) = \begin{cases}
				0, &\alpha\in(0, 1), \\
				-\frac{T}{4}\sum_{k = 1}^{m}\int_{0}^{t}\mathcal{D}^2 f(Y(s))(\sigma_k, \sigma_k) \ud s, &\alpha\in[1, \infty), 
			\end{cases}\\
			&J_{3}^{\alpha}(t) = \begin{cases}
				0, &\alpha\in(0, 1), \\
				-\frac{T}{2}\int_{0}^{t}\nabla f(Y(s))f(Y(s)) \ud s, &\alpha\in[1, \infty).
			\end{cases}
		\end{align*} 
	\end{lem}

	\begin{proof}
		The proof of the convergence in probability of $J_3^{\alpha, N}$ in $\mbf C([0, T];\mbb R^d)$ is analogous to that of $J_2^{\alpha, N}$, and thus we only present the proof for the latter.
		Denote 
		$$ \widetilde{J}_{2}^{N}(t) = -\frac{1}{2}N\sum_{k = 1}^{m}\int_{0}^{t}\mathcal{D}^2 f(Y(s))(\sigma_k, \sigma_k)(s-\kappa_N(s)) \ud s, \quad t\in[0, T]. $$ 
		For $\alpha \geq1$, we derive from Theorem \ref{TEMAconverge} and the BDG inequality that for any $t\in[0, T]$, $\mbf E|Y(t)-Y^{\alpha}_N(\kappa_N(t))|^2 \le KN^{-1}$. 
		Thus, by H\"older's inequality, we have that for any $\alpha \geq1$, 
		\begin{align} \label{ABB}
			\lim_{N\to \infty}\mbf E\|\widetilde{J}_{2}^{N}-J_{2}^{\alpha, N}\|_{\mbf C([0, T])}^2 = 0, 
		\end{align}
		due to $\mathcal{D}^2f\in \mathbf{F} $ and \eqref{AA1}.

		By H\"older's inequality, for any $p \geq 1$, there exists $K>0$ independent of $N$ such that for $t$, $s\in [0, T]$, $\mbf E|\widetilde{J}_{2}^{N}(t)-\widetilde{J}_{2}^{N}(s)|^p \le K|t-s|^p $. Then it follows from Kolmogorov's continuity theorem that $\|[\widetilde{J}_{2}^{N}]_{\mbf C^{\frac{1}{2}}([0, T])}\|_{\mbf L^4(\Omega)} \le K$, where $K$ is independent of $N $. This implies that $\mbf E\|\widetilde{J}_{2}^{N}\|_{\mbf C([0, T])}^4 \le K $. Therefore, we have that
		\begin{align} \label{AAA}
			\mbf E\|\widetilde{J}_{2}^{N}\|_{\mbf C^\alpha([0, T])}^4 \le K, 
		\end{align}
		with $\alpha = \frac{1}{2}$, where $K$ is independent of $N $. Let $Q(\varphi) = \|\varphi\|_{\mbf C^{\frac{1}{2}}([0, T])}$, $\varphi\in \mbf C^{\frac{1}{2}}([0, T];\R^d)$. By the Arzel\`{a}--Ascoli theorem, one obtains that for any $a>0$, $K_a := \{\varphi\in \mbf C([0, T];\mbb R^d):Q(\varphi) \le a\}$ is compact in $\mbf C([0, T];\mbb R^d)$. Hence, $Q$ is the Lyapunov functional as in \cite[Proposition 6.8]{da2014stochastic}. By \eqref{AAA} and H\"older's inequality, we deduce that $\mbf E Q(\widetilde{J}_{2}^{N}) = \mbf E\|\widetilde{J}_{2}^{N}\|_{\mbf C^{\frac{1}{2}}([0, T])} \le K$, where $K$ is independent of $N $. Thus, it follows from \cite[Proposition 6.8]{da2014stochastic} that the distribution of $\{\widetilde{J}_{2}^{N}\}$, i.e., $\{\mbf P\circ(\widetilde{J}_{2}^{N})^{-1}\}$, is tight in $\mbf C([0, T];\mbb R^d)$.

Let $\{\widetilde{J}_{2}^{n}\}$ and $\{\widetilde{J}_{2}^{j}\}$ be two arbitrary subsequences of $\{\widetilde{J}_{2}^{N}\}$, and let $\{\widetilde{J}_{2}^{n_k}\}$ and $\{\widetilde{J}_{2}^{j_k}\}$ be their respective subsequences. Then the distribution of $\{(\widetilde{J}_{2}^{n_k}, \widetilde{J}_{2}^{j_k}, Y)\}_{k \geq1}$ is tight in $\mbf C([0, T];\R^d)^{\otimes 3} $. Thus, by Prokhorov's theorem, there exists a subsequence $\{(\widetilde{J}_{2}^{n_k'}, \widetilde{J}_{2}^{j_k'}, Y)\}_{k \geq1}$ of $\{(\widetilde{J}_{2}^{n_k}, \widetilde{J}_{2}^{j_k}, Y)\}_{k \geq1}$ converging in distribution to a random element $Z = (Z_1, Z_2, Z_3)$ taking values in $\mbf C([0, T];\R^d)^{\otimes 3} $. It follows from Skorohod's represent theorem that there exist a probability space $(\widehat{\Omega}, \widehat{\mathcal{F}}, \widehat{\mbf P})$ and random variables $(Z_1^k, Z_2^k, Y^k)$, $\widehat{Z}$ from $(\widehat{\Omega}, \widehat{\mathcal{F}}, \widehat{\mbf P})$ to $\mbf C([0, T];\R^d)^{\otimes 3}$ such that $(Z_1^k, Z_2^k, Y^k)$ converges to $\widehat{Z} = (\widehat{Z}_1, \widehat{Z}_2, \widehat{Z}_3)$ a.s. in $(\widehat{\Omega}, \widehat{\mathcal{F}}, \widehat{\mbf P})$ as $k\to \infty$ and for any $k \geq 1$, 
		\begin{align*}
			\widehat{\mbf P}\circ (Z_1^k, Z_2^k, Y^k)^{-1} = \mbf P\circ (\widetilde{J}_{2}^{n_k'}, \widetilde{J}_{2}^{j_k'}, Y)^{-1}, \quad \widehat{\mbf P}\circ (\widehat{Z}_1, \widehat{Z}_2, \widehat{Z}_3)^{-1} = \mbf P\circ (Z_1, Z_2, Z_3)^{-1}.
		\end{align*}
		Note that 
		\begin{align*}
			&Z_1^k(t) = -\frac{1}{2}n_k'\sum_{i = 1}^{m}\int_{0}^{t}\mathcal{D}^2 f(Y^k(s))(\sigma_i, \sigma_i)(s-\kappa_{n_k'}(s)) \ud s, \\
			&Z_2^k(t) = -\frac{1}{2}j_k'\sum_{i = 1}^{m}\int_{0}^{t}\mathcal{D}^2 f(Y^k(s))(\sigma_i, \sigma_i)(s-\kappa_{j_k'}(s)) \ud s.
		\end{align*}
		Due to $\widehat{\mbf P}\circ (Y^k)^{-1} = \mbf P\circ Y^{-1}$, it holds that $\mathcal{D}^2 f(Y^k)(\sigma_i, \sigma_i)$, $k=1,2,\cdots$,	 are uniformly integrable as $\mbf C([0, T];\mbb R^d)$-valued random variables in $(\widehat{\Omega}, \widehat{\mathcal{F}}, \widehat{\mbf P}) $. Since $Y^k$ converges to $\widehat{Z}_3$ a.s. as $\mbf C([0, T];\mbb R^d)$-valued random variables, we obtain that $\mathcal{D}^2 f(Y^k)(\sigma_i, \sigma_i)$ converges to $\mathcal{D}^2 f(\widehat{Z}_3)(\sigma_i, \sigma_i)$ in $\mbf L^2(\widehat{\Omega};\mbf C([0, T];\mbb R^d))$. Then \cite[Proposition 4.2]{HJWY24} yields that for any $t\in[0, T]$, 
		\begin{align*}
			Z_1^k(t) \longrightarrow -\frac{T}{4}\sum_{i = 1}^{m}\int_{0}^{t}\mathcal{D}^2 f(\widehat{Z}_3(s))(\sigma_i, \sigma_i) \ud s, \quad Z_2^k(t) \longrightarrow -\frac{T}{4}\sum_{i = 1}^{m}\int_{0}^{t}\mathcal{D}^2 f(\widehat{Z}_3(s))(\sigma_i, \sigma_i) \ud s, 
		\end{align*}
		in $\mbf L^2(\widehat{\Omega}, \widehat{\mathcal{F}}, \widehat{\mbf P};\R^d) $. Therefore it follows that for any $t\in[0, T]$, 
		\begin{align*}
			\widehat{Z}_1(t) = -\frac{T}{4}\sum_{i = 1}^{m}\int_{0}^{t}\mathcal{D}^2 f(\widehat{Z}_3(s))(\sigma_i, \sigma_i) \ud s = \widehat{Z}_2(t) \quad \text{ a.s.\ in }(\widehat{\Omega}, \widehat{\mathcal{F}}, \widehat{\mbf P}).
		\end{align*}
Noting that $\widehat{Z}_1$ and $\widehat{Z}_2$ have continuous sample paths, we have that $\widehat{Z}_1 = \widehat{Z}_2$ a.s.\ as $\mbf C([0, T];\mbb R^d)$-valued random variables in $(\widehat{\Omega}, \widehat{\mathcal{F}}, \widehat{\mbf P}) $. Hence, the limit $(Z_1, Z_2)$ of the sequence $\{(\widetilde{J}_{2}^{n_k'}, \widetilde{J}_{2}^{j_k'})\}_{k \geq1}$ is supported on the diagonal $\{(x, y)\in \mbf C([0, T];\R^d)^{\otimes 2}:x = y\}$, due to $\mbf P\circ (Z_1, Z_2)^{-1} = \widehat{\mbf P}\circ(\widehat{Z}_1, \widehat{Z}_2)^{-1} $. We deduce from \cite[Lemma 1.1]{gyongy2022existence} that $\widetilde{J}_2^N$ converges in probability to a $\mbf C([0, T];\mbb R^d)$-valued random variable. Applying \cite[Proposition 4.2]{HJWY24}, we have that for any $t\in[0, T]$, $\widetilde{J}_2^N(t)\overset{\mbf P}{\longrightarrow}J_2(t)$, where $J_2(t) = -\frac{T}{4}\sum_{i = 1}^{m}\int_{0}^{t}\mathcal{D}^2 f(Y(s))(\sigma_i, \sigma_i) \ud s $. Therefore $\widetilde{J}_2^N\overset{\mbf P}{\longrightarrow}J_2$ as $\mbf C([0, T];\mbb R^d)$-valued random variables. Consequently, combining \eqref{ABB}, we deduce that for $\alpha \geq1$, ${J}_2^{\alpha, N}\overset{\mbf P}{\longrightarrow}J_2$ as $\mbf C([0, T];\mbb R^d)$-valued random variables. Further, for $\alpha\in(0, 1)$, it can be shown that $\mbf E\|J_2^{\alpha, N}\|^2_{\mbf C([0, T])} \le KN^{2\alpha-2} $. Thus, the proof is complete.
	\end{proof}

With previous preparation, now we can give the limit distribution of $V^{\alpha, N}(t)$. 
	
	\begin{theo} \label{maintheorem2}
		Let Assumptions \ref{assum3} and \ref{assum4} hold. Then for any $t\in[0, T]$ and $\alpha>0$, we have that $V^{\alpha, N}(t) \overset{d}{\Longrightarrow} V^\alpha(t)$ as $N\to\infty$, and thus $N^{\alpha\wedge1}(\bar{Y}^\alpha_N-Y(T)) \overset{d}{\Longrightarrow} V^\alpha(T)$. Here, $V^\alpha = \{V^\alpha(t)\}_{t\in[0, T]}$ is the strong solution of the following equation
		\begin{align}
			V^\alpha(t)
			& = \int_0^t\nabla f(Y(s))V^{\alpha}(s) \ud s-\mathbbm 1_{ \{0<\alpha \le 1\}}T^{\alpha}\int_0^tf(Y(s))|Y(s)|^{l_\alpha} \ud s\notag \\
			&\quad-\mathbbm 1_{ \{\alpha \geq1\}}\frac{T}{2}\int_{0}^{t}\nabla f(Y(s))f(Y(s)) \ud s-\mathbbm 1_{ \{\alpha \geq1\}}\frac{T}{2}\int_{0}^{t}\nabla f(Y(s))\sigma \ud \mathbf{W}(s)\notag\\
			&\quad-\mathbbm 1_{ \{\alpha \geq 1\}}\frac{T}{4}\sum_{k = 1}^{m}\int_{0}^{t}\mathcal{D}^2 f(Y(s))(\sigma_k, \sigma_k) \ud s-\mathbbm 1_{ \{\alpha \geq 1\}}\frac{T}{\sqrt{12}}\int_{0}^{t}\nabla f(Y(s))\sigma \ud \widehat{\mathbf W}(s), 
		\end{align}
		where $l_\alpha = \lceil 2\alpha\rceil l$, and $\mbf {\widehat{W}} = (\widehat{W}_{1}, \widehat{W}_{2}, \ldots, \widehat{W}_{m})$ is an $m$-dimensional standard Brownian motion independent of $\mbf W$.
	\end{theo}

	\begin{proof}
		Similar to Lemma \ref{UN-UtildeN}, we can derive that for $\alpha>0$, $J_0^{\alpha, N}$ converges to $J_0^\alpha$ as $\mbf C([0, T];\mbb R^d)$-valued random variables in $\mbf L^2(\Omega, \mathcal{F}, \mbf P;\R^d)$ with 
		\begin{align*}
			J_0^\alpha(t) = \begin{cases}
				-T^{\alpha}\int_0^tf(Y(s))|Y(s)|^{l_\alpha} \ud s, &\alpha\in(0, 1], \\
				0, &\alpha\in (1, + \infty).
			\end{cases}
		\end{align*} 
		Denote $J^N(t) = -T\sum_{k = 1}^{m}\int_{0}^{t}\nabla f(Y^\alpha_N(\kappa_N(s)))\sigma_k(\lfloor\frac{Ns}{T}\rfloor + 1-\frac{Ns}{T}) \ud W_k(s)$, $t\in[0, T] $. Let $J^{N, i}$ and $(\nabla f(Y(s))\sigma_k)^i$ denote the $i$th component of $J^N$ and $\nabla f(Y(s))\sigma_k$ respectively, $i = 1, \ldots, d$, $k = 1, \ldots, m $. Applying \cite[Proposition 4.2]{HJWY24} yields that for any $t\in [0, T]$, $\LL J^{N, i}, W_j\RR_t$ converges to $-\frac{T}{2}\int_{0}^{t}(\nabla f(Y(s))\sigma_j)^i \ud s$ in $\mbf L^2(\Omega, \mathcal{F}, \mbf P;\R^d)$, $i = 1, \ldots, d$, $j = 1, \ldots, m $. Similar to the proof of \eqref{cross2}, we derive that for $i$, $j = 1, \ldots, d$, 
		\begin{align*}
			\LL J^{N, i}, J^{N, j}\RR_t \longrightarrow \frac{T^2}{3}\sum_{k = 1}^{m}\int_{0}^{t}\big(\nabla f(Y(s))\sigma_k\big )^i\big(\nabla f(Y(s))\sigma_k\big )^j \ud s, \quad \text{ in $\mbf L^2(\Omega, \mathcal{F}, \mbf P) $. }
		\end{align*}
		Further, following the ideas for proving Lemma \ref{IN41}, it can be shown that $J^{N}\overset{stably}{\Longrightarrow}J$ as $\mbf C([0, T];\mbb R^d)$-valued random variables as $N\to\infty$, where
		\begin{align*}
			J(t) := -\frac{T}{2}\int_{0}^{t}\nabla f(Y(s))\sigma \ud \mathbf{W}(s)-\frac{T}{\sqrt{12}}\int_{0}^{t}\nabla f(Y(s))\sigma \ud \widehat{\mathbf W}(s), \quad \text{ $t\in[0, T] $. }
		\end{align*}
		Further, for $\alpha\in(0, 1)$, one can obtain that $\mbf E\|J_1^{\alpha, N}\|_{\mbf C([0, T])}^2 \le KN^{2\alpha-2} $. Therefore, $J_1^{\alpha, N}\overset{stably}{\Longrightarrow}J_1^\alpha$ as $\mbf C([0, T];\mbb R^d)$-valued random variables as $N\to\infty$, where
		\begin{align*}
			J_{1}^{\alpha}(t) = \begin{cases}
				0, &\alpha\in(0, 1), \\
				J(t), &\alpha\in[1, \infty).
			\end{cases}
		\end{align*}
		Following the argument for proving Lemma \ref{A1} and Theorem \ref{maintheorem1}, and using Lemma \ref{J2C}, we finally complete the proof.
	\end{proof}

	\begin{cor} \label{Cor2}
		Let Assumptions \ref{assum3} and \ref{assum4} hold. Then there exist $C_4$, $C_5$, $C_6>0$ independent of T such that for $\alpha \geq 1$ and $t\in[0, T]$, 
		\begin{align}
			\mbf E|V^\alpha(t)|^2 \le (C_4 + C_5\mathbbm1_{\{\alpha = 1\}})e^{C_6T}T^3.
		\end{align}
	\end{cor}

	\begin{rem} \label{RB}
		As noted in Remark \ref{RA}, Theorem \ref{maintheorem2} and Corollary \ref{Cor2} indicate that $\alpha$ is the key parameter reflecting the growth rate of the mean-square error for the tamed Euler method \eqref{TEMA}. In addition, we infer that for the tamed Euler method \eqref{TEMA} of strong order $1$, the one with $\alpha = 1$ has the largest mean-square error after a long time.
	\end{rem}

	\section{Numerical experiments} \label{Sec5}
	In this section, we perform numerical experiments to validate our theoretical analysis. We verify the strong convergence order of the tamed Euler method for both multiplicative and additive noise and verify Corollaries \ref{Cor1} and \ref{Cor2} by presenting the evolution of the mean-square error of tamed Euler methods w.r.t.\ time.

	\subsection{Numerical experiments for multiplicative noise}
	
	\begin{ex}
		Consider the following SDE:
		\begin{align} \label{SDEE1}
			\begin{cases}
				 \ud X(t) = -\big (X(t)\big )^5 \ud t + X(t) \ud W_1(t) + \big (X(t)\big )^2 \ud W_2(t), \qquad t\in(0, T], \\
				X(0) = X_0\in\mbb R, 
			\end{cases}
		\end{align}
		where $W_1$ and $W_2$ are two independent one-dimensional standard Brownian motions.
	\end{ex}

	In this experiment, we consider the tamed Euler methods \eqref{TEM} with $\alpha = 0.2$, $0.4$, $0.5$, and $1$ for approximating \eqref{SDEE1}. It is clear that Assumption \ref{assum1} is satisfied by the coefficients of \eqref{SDEE1} with $l = 4$. First, we test the mean-square convergence order for the above tamed Euler methods. We set $X_0 = 1$ and $T = 1 $. The exact solution $X(T)$ is approximated by the tamed Euler methods with the small step-size $h = 2^{-16} $. The expectation is obtained based on the Monte Carlo method with 1000 sample paths. One can observe in Figure \ref{F1} that the tamed Euler method \eqref{TEM} has mean-square convergence of order $\alpha\wedge\frac{1}{2}$.

	\begin{figure}[H]
		\centering
		\subfigure[]{\includegraphics[width = 0.48\textwidth]{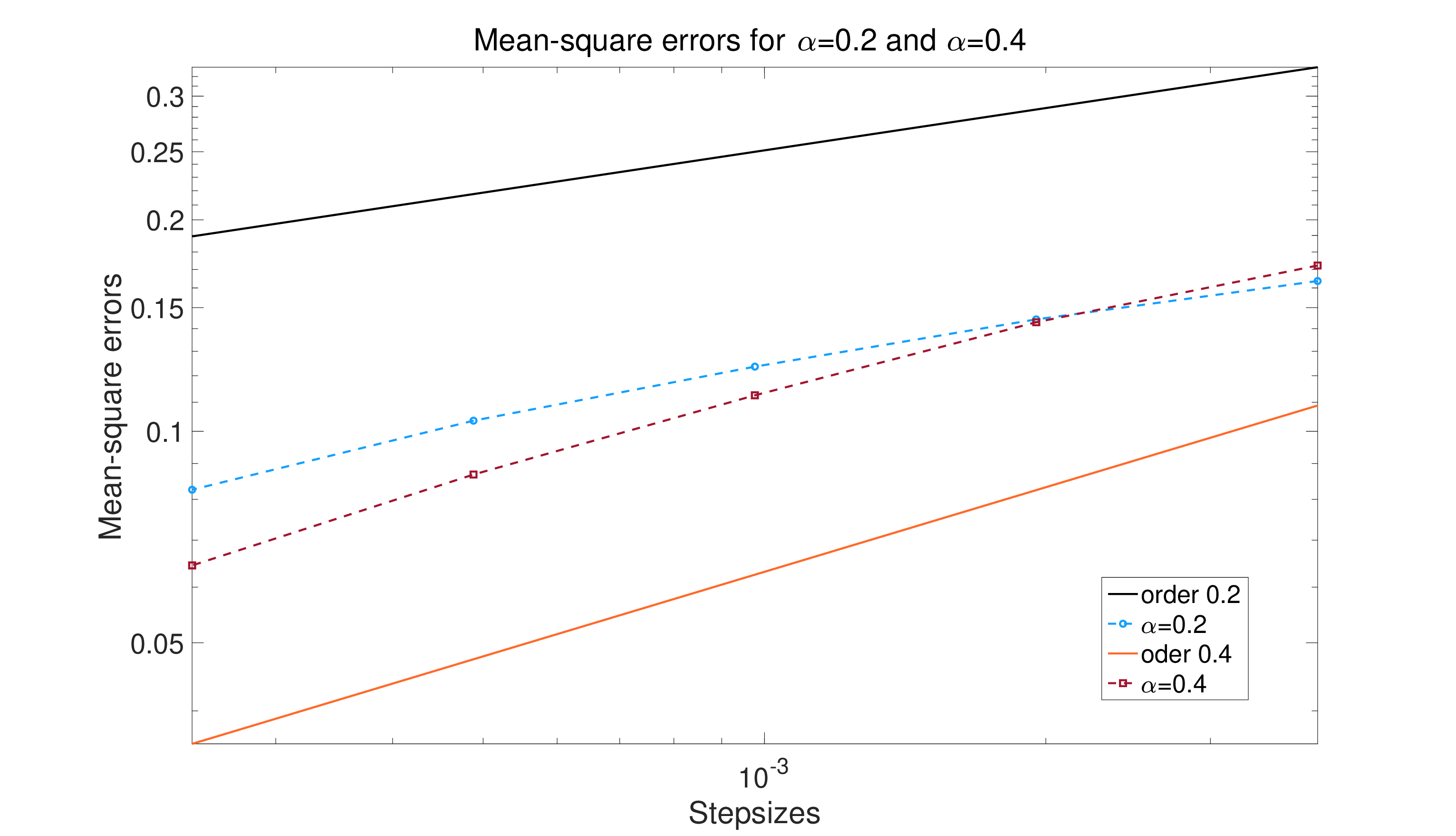}}
		\subfigure[]{\includegraphics[width = 0.48\textwidth]{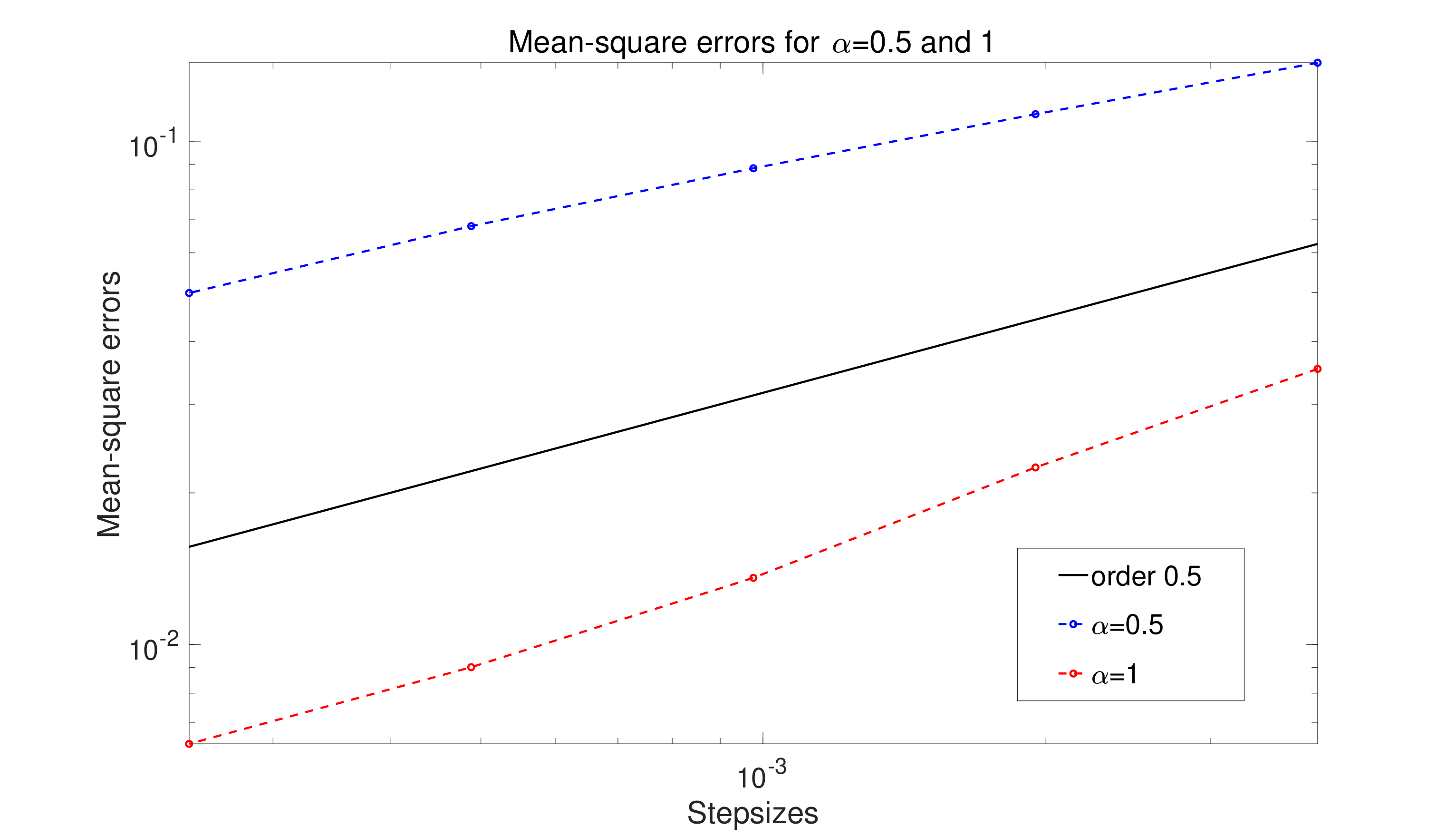}}
		\caption{Mean square errors for tamed Euler methods \eqref{TEM} with $\alpha = 0.2$, $0.4$, $0.5$, and $1$ applied to \eqref{SDEE1} in the log-log scale for five different step-sizes $h = 2^{-8}, 2^{-9}, 2^{-10}, 2^{-11}, 2^{-12}$.} \label{F1}
	\end{figure}

	Then we test the evolution of mean-square errors of tamed Euler methods \eqref{TEM} with $\alpha = 0.5$, $0.7$, $0.8$, and $1$ w.r.t.\  the time $t $. The exact solution is approximated by the tamed Euler methods with the small step-size $h = 10^{-4} $. We fix $X_0 = 1$ and use 1000 sample paths to approximate the expectation. As is shown in Figure \ref{F3}, the tamed Euler method \eqref{TEM} with $\alpha = 0.5$ has the largest mean-square error among these four methods after a long time. This verifies Corollary \ref{Cor1} and Remark \ref{RA}. Additionally, as the value of $\alpha$ increases, the corresponding mean-square error decreases accordingly.

	\begin{figure}[H]
		\centering
		\subfigure[]{\includegraphics[width = 0.48\textwidth]{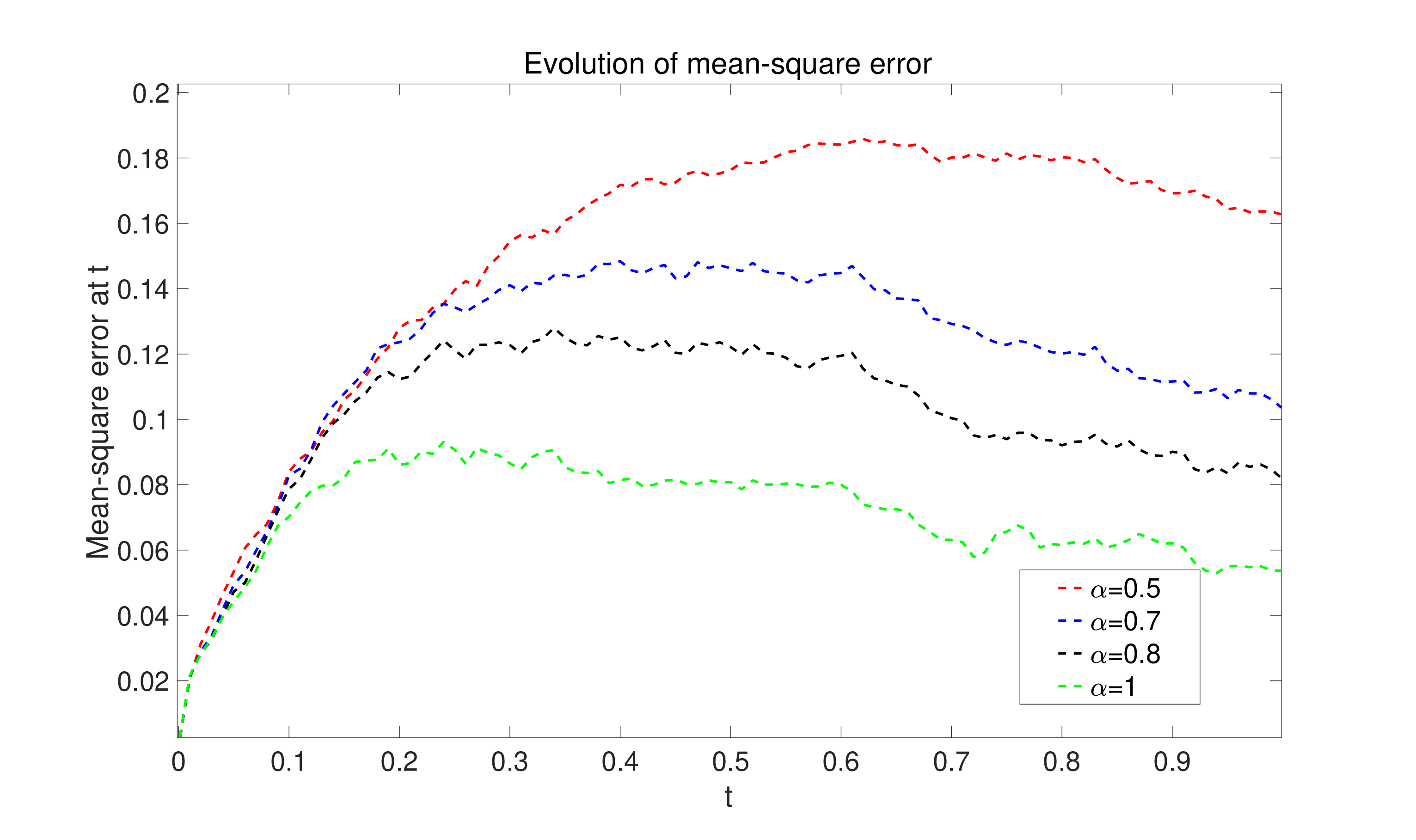}}
		\subfigure[]{\includegraphics[width = 0.48\textwidth]{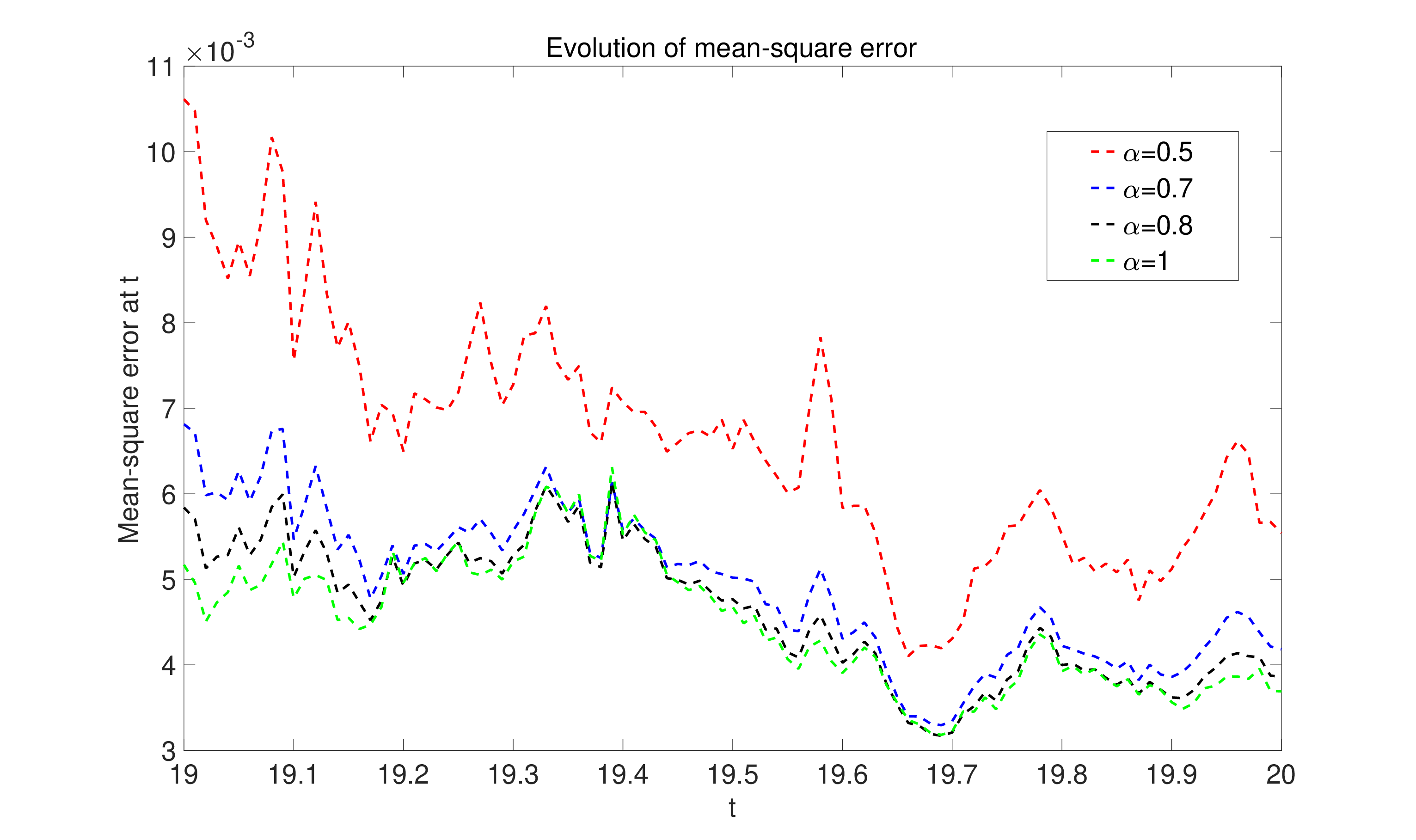}}
		\caption{Evolution of mean-square errors for tamed Euler methods \eqref{TEM} with $\alpha = 0.5$, $0.7$, $0.8$, and $1$ applied to \eqref{SDEE1} with $h = 10^{-2}$.} \label{F3}
	\end{figure}

	\subsection{Numerical experiments for additive noise}
	
	\begin{ex}
		Consider the following SDE:
		\begin{align} \label{SDEE2}
			\begin{cases}
				 \ud Y(t) = -\big (Y(t)\big )^3 \ud t + \sigma \ud W(t), \qquad t\in(0, T], \\
				Y(0) = Y_0\in\mbb R, 
			\end{cases}
		\end{align}
		where $W$ is a one-dimensional standard Brownian motion, and $\sigma$ is a given constant.
	\end{ex}

	In this experiment for additive noise, we adopt the tamed Euler method \eqref{TEMA} with $\alpha = 0.2$, $0.7$, $1$, $1.5$, and $2$ to approximate the solution of \eqref{SDEE2}. It can be shown that Assumption \ref{assum3} is fulfilled by the coefficients of \eqref{SDEE2} with $l = 2$. We test the mean-square convergence orders of these methods by setting $Y_0 = 1$, $T = 1$, and $\sigma = 1 $. The exact solution $Y(T)$ is approximated by the respective tamed Euler methods using the small step-size $h = 2^{-14} $. The expectation is obtained based on the Monte Carlo method with 1000 sample paths. As illustrated in Figure \ref{F2}, the tamed Euler method \eqref{TEMA} achieves a mean-square convergence order of $\alpha\wedge1$. 
	\begin{figure}[H]
		\centering
		\subfigure[]{\includegraphics[width = 0.48\textwidth]{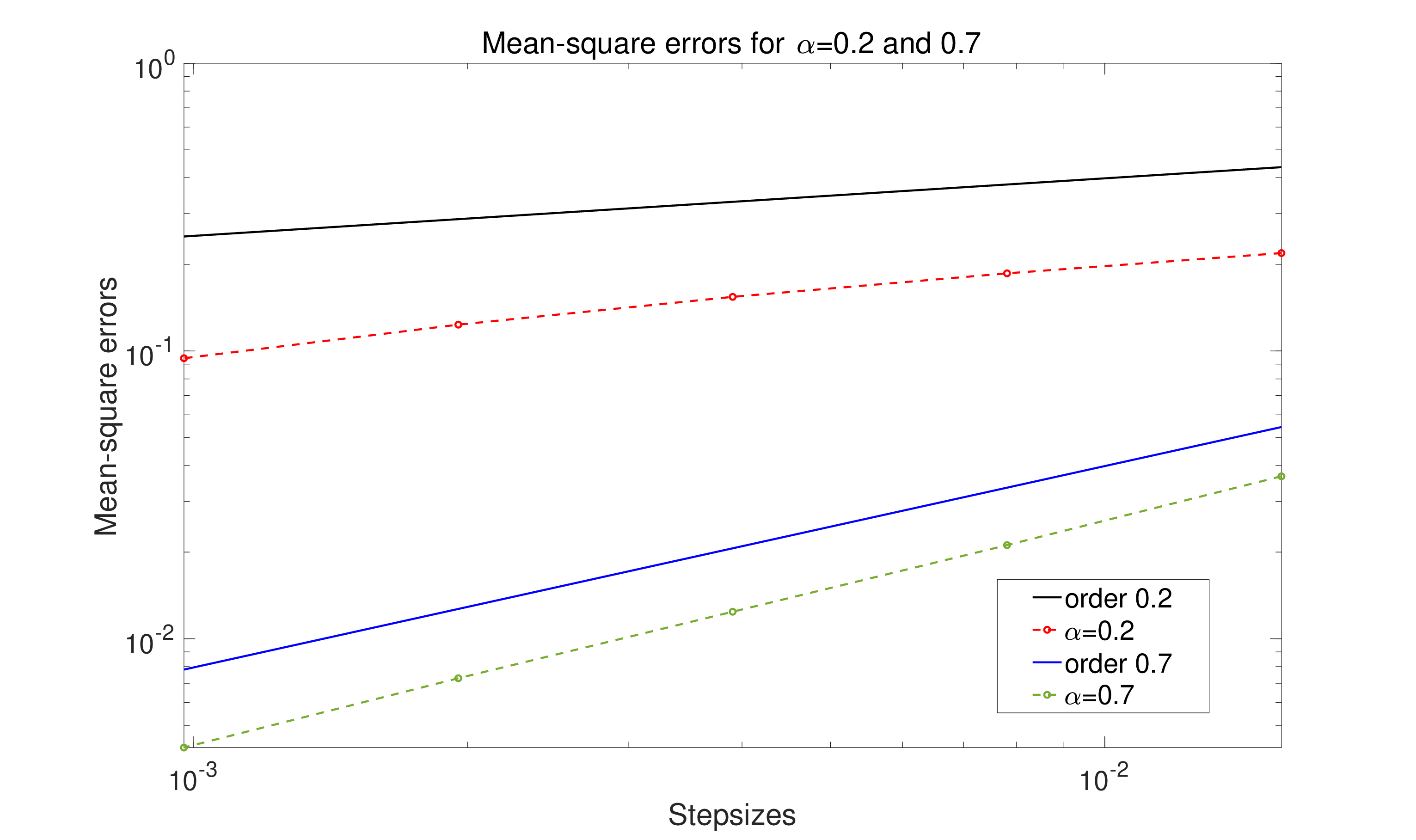}}
		\subfigure[]{\includegraphics[width = 0.48\textwidth]{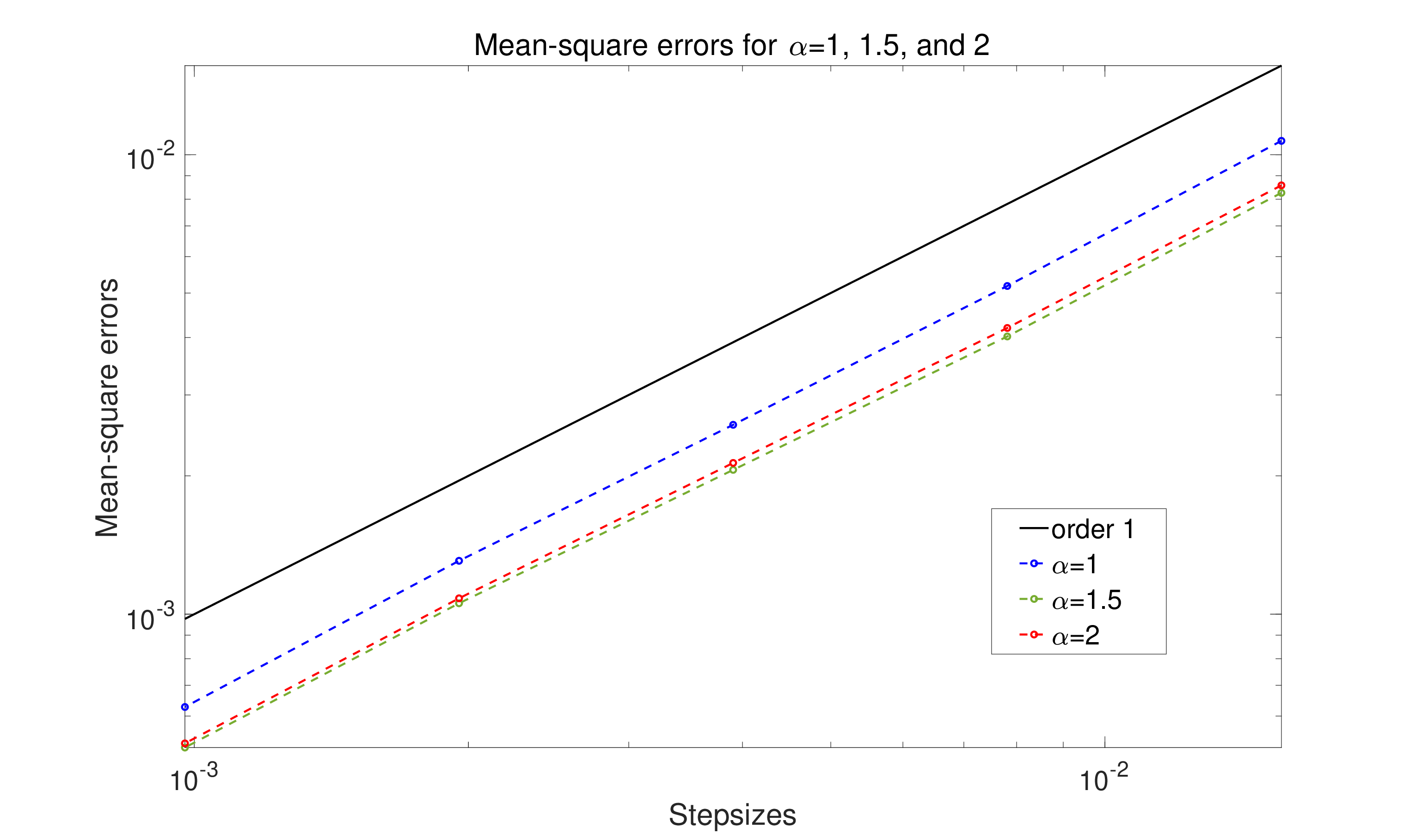}}
		\caption{Mean square errors for tamed Euler methods \eqref{TEMA} with $\alpha = 0.2$, $0.7$, $1$, $1.5$, and $2$ applied to \eqref{SDEE2} in the log-log scale for five different step-sizes $h = 2^{-6}, 2^{-7}, 2^{-8}, 2^{-9}, 2^{-10}$.} \label{F2}
	\end{figure}
	Finally, we test the evolution of mean-square errors of tamed Euler methods \eqref{TEMA} with $\alpha = 1$, $1.5$, $2$, and $2.5$ w.r.t.\  the time $t $. The exact solution is approximated by the respective tamed Euler methods with the small step-size $h = 10^{-4} $. We take $Y_0 = 1$ and $\sigma = 1 $. The expectation is computed by the Monte Carlo method with 1000 sample paths. Figure \ref{F4} reveals that the tamed Euler method \eqref{TEMA} with $\alpha = 1$ has the largest mean-square error after a long time, which verifies Corollary \ref{Cor2} and Remark \ref{RB}.

	\begin{figure}[H]
		\centering
		\subfigure[]{\includegraphics[width = 0.48\textwidth]{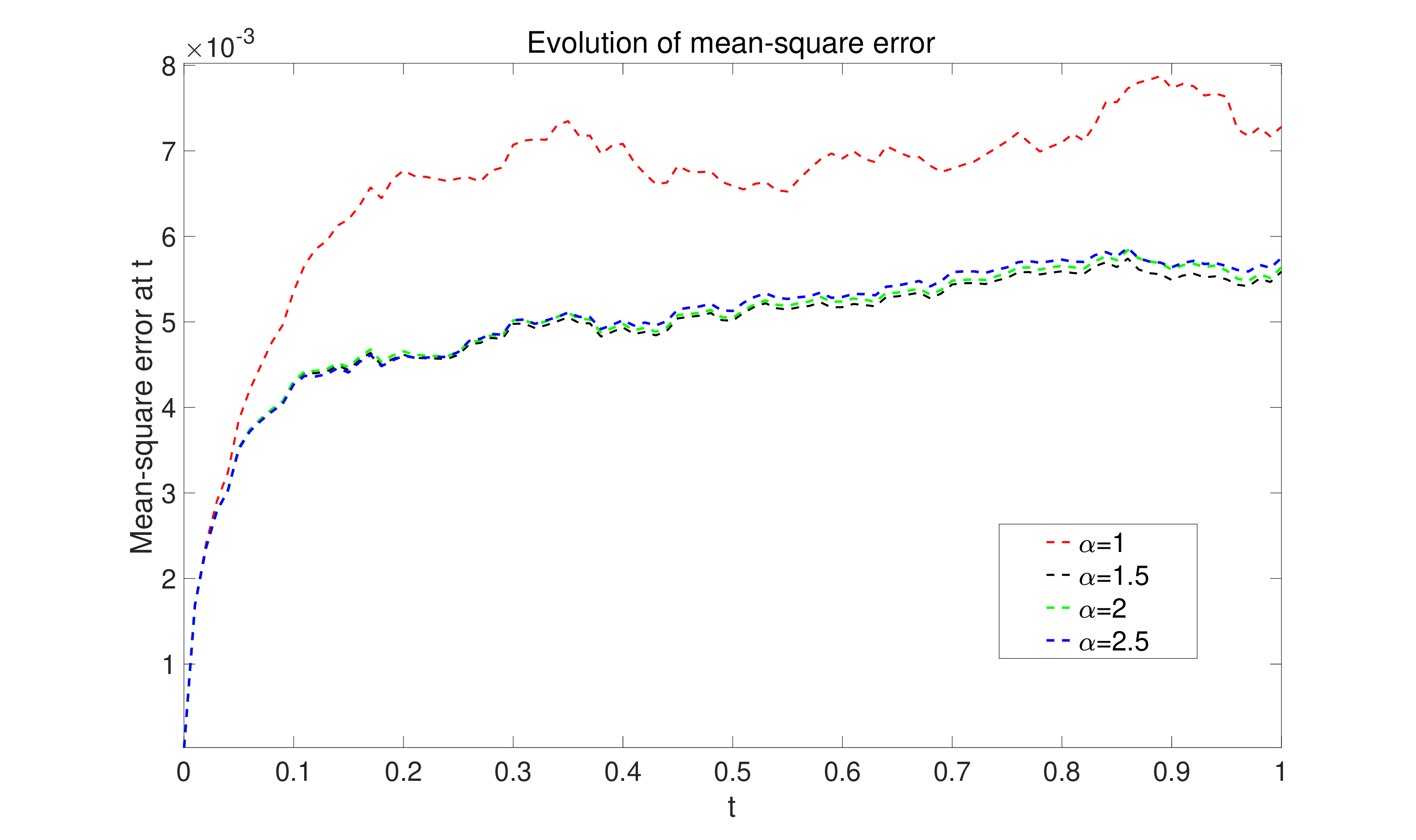}}
		\subfigure[]{\includegraphics[width = 0.48\textwidth]{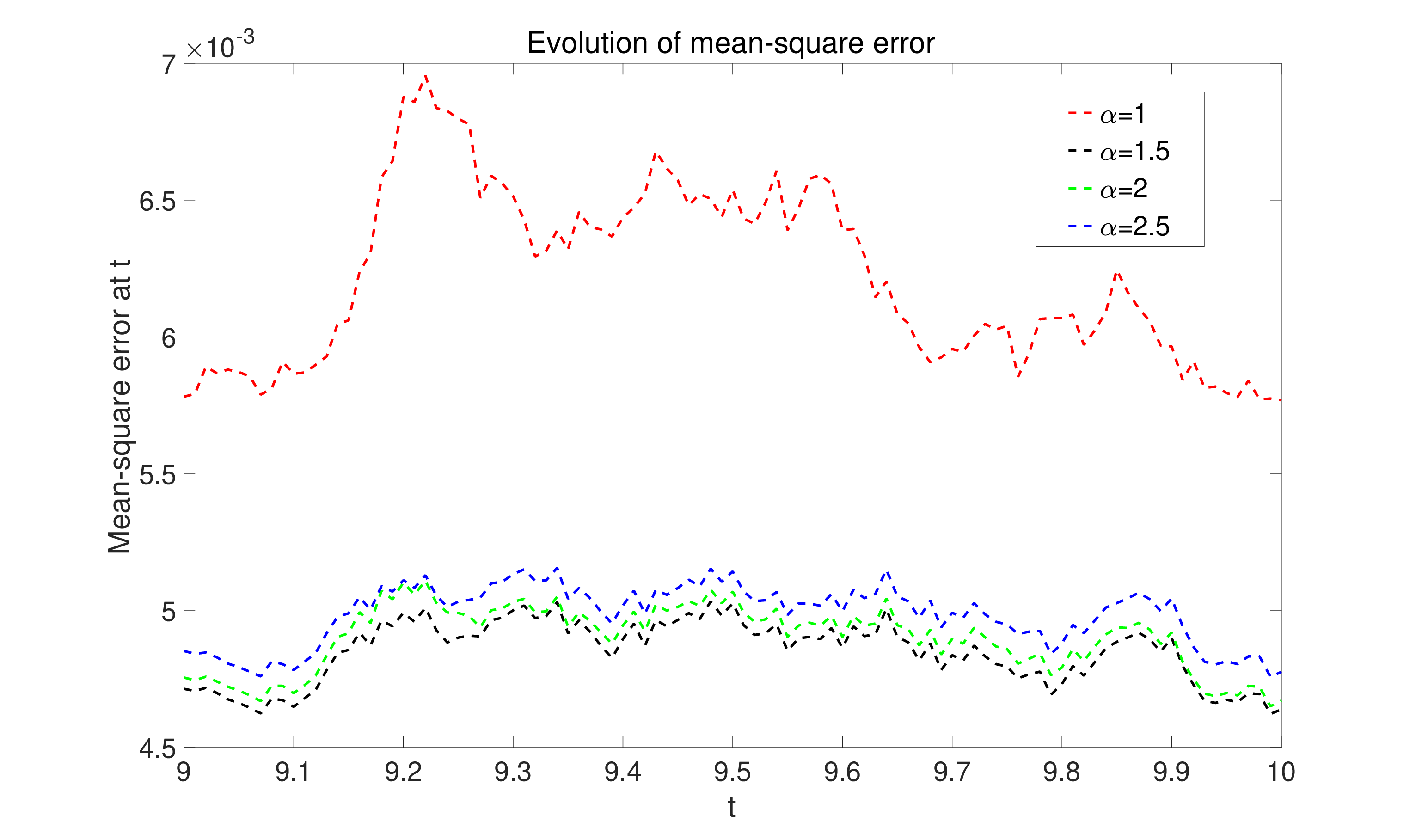}}
		\caption{Evolution of mean-square errors for tamed Euler methods \eqref{TEMA} with $\alpha = 1$, $1.5$, $2$, and $2.5$ applied to \eqref{SDEE2} with $h = 10^{-2}$.}
		 \label{F4}
	\end{figure}

	\bibliographystyle{plain}
	\bibliography{mybibfile}

@article {Sabanis13,
	AUTHOR = {Sabanis, S.},
	TITLE = {A note on tamed {E}uler approximations},
	JOURNAL = {Electron. Commun. Probab.},
	FJOURNAL = {Electronic Communications in Probability},
	VOLUME = {18},
	YEAR = {2013},
	PAGES = {no. 47, 10},
	ISSN = {1083-589X},
	MRCLASS = {60H35},
	MRNUMBER = {3070913},
	DOI = {10.1214/ECP.v18-2824},
	URL = {https://doi.org/10.1214/ECP.v18-2824},
}

@article {Bouleau,
	AUTHOR = {Bouleau, N.},
	TITLE = {When and how an error yields a {D}irichlet form},
	JOURNAL = {J. Funct. Anal.},
	FJOURNAL = {Journal of Functional Analysis},
	VOLUME = {240},
	YEAR = {2006},
	NUMBER = {2},
	PAGES = {445--494},
	ISSN = {0022-1236},
	MRCLASS = {60J45},
	MRNUMBER = {2261691},
	DOI = {10.1016/j.jfa.2006.03.007},
	URL = {https://doi.org/10.1016/j.jfa.2006.03.007},
}

@article {CLTMonte15,
	AUTHOR = {Ben Alaya, M. and Kebaier, A.},
	TITLE = {Central limit theorem for the multilevel {M}onte {C}arlo
	{E}uler method},
	JOURNAL = {Ann. Appl. Probab.},
	FJOURNAL = {The Annals of Applied Probability},
	VOLUME = {25},
	YEAR = {2015},
	NUMBER = {1},
	PAGES = {211--234},
	ISSN = {1050-5164},
	MRCLASS = {60F05 (60H35 62F12 65C05)},
	MRNUMBER = {3297771},
	MRREVIEWER = {Charles-Edouard Br\'{e}hier},
	DOI = {10.1214/13-AAP993},
	URL = {https://doi.org/10.1214/13-AAP993},
}

@article {MAM,
	AUTHOR = {Hong, J. and Jin, D. and Sheng, D.},
	TITLE = {Convergence analysis for minimum action methods coupled with a
	finite difference method},
	JOURNAL = {IMA J. Numer. Anal.},
	FJOURNAL = {IMA Journal of Numerical Analysis},
	VOLUME = {45},
	YEAR = {2025},
	NUMBER = {3},
	PAGES = {1501--1536},
	ISSN = {0272-4979,1464-3642},
	MRCLASS = {65M12 (65C30 65M06)},
	MRNUMBER = {4917678},
	DOI = {10.1093/imanum/drae038},
	URL = {https://doi.org/10.1093/imanum/drae038},
}

@book {Maoxuerong,
	AUTHOR = {Mao, X.},
	TITLE = {Stochastic {D}ifferential {E}quations and {A}pplications},
	EDITION = {Second},
	PUBLISHER = {Horwood Publishing Limited, Chichester},
	YEAR = {2008},
	PAGES = {xviii+422},
	ISBN = {978-1-904275-34-3},
	MRCLASS = {60-02 (34F05 34K50 60H10 60H30 91B28)},
	MRNUMBER = {2380366},
	DOI = {10.1533/9780857099402},
	URL = {https://doi.org/10.1533/9780857099402},
}

@article {Protter2020SPA,
	AUTHOR = {Protter, P. and Qiu, L. and Martin, J. S.},
	TITLE = {Asymptotic error distribution for the {E}uler scheme with
	locally {L}ipschitz coefficients},
	JOURNAL = {Stochastic Process. Appl.},
	FJOURNAL = {Stochastic Processes and their Applications},
	VOLUME = {130},
	YEAR = {2020},
	NUMBER = {4},
	PAGES = {2296--2311},
	ISSN = {0304-4149},
	MRCLASS = {60H35 (60F05 60H10)},
	MRNUMBER = {4074703},
	MRREVIEWER = {C\'{o}nall Kelly},
	DOI = {10.1016/j.spa.2019.07.003},
	URL = {https://doi.org/10.1016/j.spa.2019.07.003},
}

@article {Protter1998AOP,
	AUTHOR = {Jacod, J. and Protter, P.},
	TITLE = {Asymptotic error distributions for the {E}uler method for
	stochastic differential equations},
	JOURNAL = {Ann. Probab.},
	FJOURNAL = {The Annals of Probability},
	VOLUME = {26},
	YEAR = {1998},
	NUMBER = {1},
	PAGES = {267--307},
	ISSN = {0091-1798},
	MRCLASS = {60H10 (60G44 65U05)},
	MRNUMBER = {1617049},
	MRREVIEWER = {Denis Talay},
	DOI = {10.1214/aop/1022855419},
	URL = {https://doi.org/10.1214/aop/1022855419},
}

@book {Kloeden1992,
	AUTHOR = {Kloeden, P. E. and Platen, E.},
	TITLE = {Numerical {S}olution of {S}tochastic {D}ifferential {E}quations},
	SERIES = {Applications of Mathematics (New York)},
	VOLUME = {23},
	PUBLISHER = {Springer-Verlag, Berlin},
	YEAR = {1992},
	PAGES = {xxxvi+632},
	ISBN = {3-540-54062-8},
	MRCLASS = {60H10 (34A50 34F05 65L99 65P05)},
	MRNUMBER = {1214374},
	MRREVIEWER = {G. N. Mil\cprime shte\u{\i}n},
	DOI = {10.1007/978-3-662-12616-5},
	URL = {https://doi.org/10.1007/978-3-662-12616-5},
}

@article {Fukasawa2023,
	AUTHOR = {Fukasawa, M. and Ugai, T.},
	TITLE = {Limit distributions for the discretization error of stochastic
	{V}olterra equations with fractional kernel},
	JOURNAL = {Ann. Appl. Probab.},
	FJOURNAL = {The Annals of Applied Probability},
	VOLUME = {33},
	YEAR = {2023},
	NUMBER = {6B},
	PAGES = {5071--5110},
	ISSN = {1050-5164},
	MRCLASS = {60H20 (60F17)},
	MRNUMBER = {4677728},
	DOI = {10.1214/23-aap1941},
	URL = {https://doi.org/10.1214/23-aap1941},
}

@book {Milsteinbook,
	AUTHOR = {Milstein, G. N. and Tretyakov, M. V.},
	TITLE = {Stochastic {N}umerics for {M}athematical {P}hysics},
	SERIES = {Scientific Computation},
	PUBLISHER = {Springer-Verlag, Berlin},
	YEAR = {2004},
	PAGES = {xx+594},
	ISBN = {3-540-21110-1},
	MRCLASS = {60-02 (35R60 37M15 60H10 60H15 60H35 65C30)},
	MRNUMBER = {2069903},
	MRREVIEWER = {Peter E. Kloeden},
	DOI = {10.1007/978-3-662-10063-9},
	URL = {https://doi.org/10.1007/978-3-662-10063-9},
}

@article {ZhangZQ13,
	AUTHOR = {Tretyakov, M. V. and Zhang, Z.},
	TITLE = {A fundamental mean-square convergence theorem for {SDE}s with
	locally {L}ipschitz coefficients and its applications},
	JOURNAL = {SIAM J. Numer. Anal.},
	FJOURNAL = {SIAM Journal on Numerical Analysis},
	VOLUME = {51},
	YEAR = {2013},
	NUMBER = {6},
	PAGES = {3135--3162},
	ISSN = {0036-1429},
	MRCLASS = {60H35 (60H10 65C30)},
	MRNUMBER = {3129758},
	MRREVIEWER = {Peter E. Kloeden},
	DOI = {10.1137/120902318},
	URL = {https://doi.org/10.1137/120902318},
}

@article{HJWY24,
	AUTHOR={Hong, J. and Jin, D. and Wang, X. and Yang, G.},
	TITLE={Asymptotic error distribution of accelerated exponential {E}uler method for     
	parabolic {SPDE}s},
	NOTE={Preprint, arXiv:\ 2409.13827},
}

@incollection {Jacod97,
	AUTHOR = {Jacod, J.},
	TITLE = {On {C}ontinuous {C}onditional {G}aussian {M}artingales and {S}table
	{C}onvergence in {L}aw},
	BOOKTITLE = {S\'{e}minaire de {P}robabilit\'{e}s, {XXXI}},
	SERIES = {Lecture Notes in Math.},
	VOLUME = {1655},
	PAGES = {232--246},
	PUBLISHER = {Springer, Berlin},
	YEAR = {1997},
	MRCLASS = {60G15 (60F05 60G44 60G48)},
	MRNUMBER = {1478732},
	MRREVIEWER = {Ireneusz Szyszkowski},
	DOI = {10.1007/BFb0119308},
	URL = {https://doi.org/10.1007/BFb0119308},
}

@article {Sabanis2016,
	AUTHOR = {Sabanis, S.},
	TITLE = {Euler approximations with varying coefficients:\ the case of
	superlinearly growing diffusion coefficients},
	JOURNAL = {Ann. Appl. Probab.},
	FJOURNAL = {The Annals of Applied Probability},
	VOLUME = {26},
	YEAR = {2016},
	NUMBER = {4},
	PAGES = {2083--2105},
	ISSN = {1050-5164,2168-8737},
	MRCLASS = {60H35 (65C30)},
	MRNUMBER = {3543890},
	MRREVIEWER = {Minoo\ Kamrani},
	DOI = {10.1214/15-AAP1140},
	URL = {https://doi.org/10.1214/15-AAP1140},
}

@incollection {Gylemma,
	AUTHOR = {Gy\"ongy, I. and Krylov, N.},
	TITLE = {On the rate of convergence of splitting-up approximations for
	{SPDE}s},
	BOOKTITLE = {Stochastic inequalities and applications},
	SERIES = {Progr. Probab.},
	VOLUME = {56},
	PAGES = {301--321},
	PUBLISHER = {Birkh\"auser, Basel},
	YEAR = {2003},
	ISBN = {3-7643-2197-0},
	MRCLASS = {65C30 (60H15 60H35 93E11)},
	MRNUMBER = {2073438},
	MRREVIEWER = {Vivek\ S.\ Borkar},
}

@article {Hutzen2012,
	AUTHOR = {Hutzenthaler, M. and Jentzen, A. and Kloeden, P.
	E.},
	TITLE = {Strong convergence of an explicit numerical method for {SDE}s
	with nonglobally {L}ipschitz continuous coefficients},
	JOURNAL = {Ann. Appl. Probab.},
	FJOURNAL = {The Annals of Applied Probability},
	VOLUME = {22},
	YEAR = {2012},
	NUMBER = {4},
	PAGES = {1611--1641},
	ISSN = {1050-5164,2168-8737},
	MRCLASS = {65C30 (60H35)},
	MRNUMBER = {2985171},
	MRREVIEWER = {Andreas\ R\"o\ss ler},
	DOI = {10.1214/11-AAP803},
	URL = {https://doi.org/10.1214/11-AAP803},
	
}

@article {SDR2025,
	AUTHOR = {Hong, J. and Liang, G. and Sheng, D.},
	TITLE = {Asymptotic error distributions of symplectic and
	non-symplectic methods for stochastic {H}amiltonian system
	with additive noise},
	JOURNAL = {Discrete Contin. Dyn. Syst.},
	FJOURNAL = {Discrete and Continuous Dynamical Systems},
	VOLUME = {48},
	YEAR = {2026},
	PAGES = {447--468},
	ISSN = {1078-0947,1553-5231},
	MRCLASS = {65C30 (60H35 65P10)},
	MRNUMBER = {4976471},
	DOI = {10.3934/dcds.2025148},
	URL = {https://doi.org/10.3934/dcds.2025148},
}

@article {Huzen2011,
	AUTHOR = {Hutzenthaler, M. and Jentzen, A. and Kloeden, P.
	E.},
	TITLE = {Strong and weak divergence in finite time of {E}uler's method
	for stochastic differential equations with non-globally
	{L}ipschitz continuous coefficients},
	JOURNAL = {Proc. R. Soc. Lond. Ser. A Math. Phys. Eng. Sci.},
	FJOURNAL = {Proceedings of The Royal Society of London. Series A.
	Mathematical, Physical and Engineering Sciences},
	VOLUME = {467},
	YEAR = {2011},
	NUMBER = {2130},
	PAGES = {1563--1576},
	ISSN = {1364-5021,1471-2946},
	MRCLASS = {65C30 (60H10 60H35)},
	MRNUMBER = {2795791},
	MRREVIEWER = {Henri\ Schurz},
	DOI = {10.1098/rspa.2010.0348},
	URL = {https://doi.org/10.1098/rspa.2010.0348},
}

@book {KolomoContinuous,
	AUTHOR = {Khoshnevisan, D.},
	TITLE = {Analysis of {S}tochastic {P}artial {D}ifferential {E}quations},
	SERIES = {CBMS Regional Conference Series in Mathematics},
	VOLUME = {119},
	PUBLISHER = {Conference Board of the Mathematical Sciences, Washington, DC;
	by the American Mathematical Society, Providence, RI},
	YEAR = {2014},
	PAGES = {viii+116},
	ISBN = {978-1-4704-1547-1},
	MRCLASS = {60H15 (35R60 60H30)},
	MRNUMBER = {3222416},
	MRREVIEWER = {Sergey\ V.\ Lototsky},
	DOI = {10.1090/cbms/119},
	URL = {https://doi.org/10.1090/cbms/119},
}

@article{Jin2025,
	title={Asymptotic error distribution for stochastic {R}unge--{K}utta methods of strong order one},
	author={Jin, D.},
	Note={Prepint, arXiv:\ 2506.08937},

}

@book{da2014stochastic,
	title={Stochastic {E}quations in {I}nfinite {D}imensions},
	author={Da Prato, G. and Zabczyk, J.},
	volume={152},
	year={2014},
	publisher={Cambridge university press}
}

@article{gyongy2022existence,
	title={Existence of strong solutions for {I}t{\^o}’s stochastic equations via approximations:\ revisited},
	author={Gy{\"o}ngy, I. and Krylov, N.V.},
	journal={Stochastics and Partial Differential Equations:\ Analysis and Computations},
	volume={10},
	number={3},
	pages={693--719},
	year={2022},
	publisher={Springer}
}
	
\end{document}